\documentclass[12pt, reqno]{amsart}
\usepackage{amssymb, amsmath, amsthm, graphicx, mathrsfs}
\usepackage[colorlinks=true, citecolor=blue, urlcolor=blue]{hyperref}

\newtheorem{thrm}{Theorem}
\newtheorem{lemma}{Lemma}
\newtheorem{probl}{Problem}

\title[A polyhedron without self-intersections, all of whose dihedral angles change]
{A flexible polyhedron without self-intersections in Euclidean 3-space, 
all of whose dihedral angles change during a flex}
\author{V.A. Alexandrov}
\address{Sobolev Institute of Mathematics, Koptyug ave., 4, Novosibirsk, 630090, 
Russia and Department of Physics, Novosibirsk State University, Pirogov str., 2, Novosibirsk, 630090, Russia\newline 
\indent \href{https://orcid.org/0000-0002-6622-8214}{\rm{ORCID iD: 0000-0002-6622-8214}}}
\email{alex@math.nsc.ru}
\author{E.P. Volokitin}
\address{Sobolev Institute of Mathematics, Koptyug ave., 4, Novosibirsk, 630090, 
Russia and Department of Physics, Novosibirsk State University, Pirogov str., 2, Novosibirsk, 630090, Russia\newline
\indent \href{https://orcid.org/0000-0002-2646-7800}{\rm{ORCID iD: 0000-0002-2646-7800}}}
\email{volok@math.nsc.ru}

\begin{document}

\begin{abstract}
We construct a sphere-homeomorphic flexible self-intersection free polyhedron 
in Euclidean 3-space such that all its dihedral angles change during some flex 
of this polyhedron.
The constructed polyhedron has 26 vertices, 72 edges and 48 faces. 
To study its properties, we use symbolic computations in the Wolfram Mathematica 
software system.
\par
\textit{UDK}: 514.1
\end{abstract}
\maketitle
\section{Introduction}\label{sec1}

In this article, a compact polyhedral surface in Euclidean 3-space $\mathbb{R}^3$
with or without boundary, all of whose faces are triangles, and which a priori 
can have self-intersections of any type is called a \textit{polytope}.

A polyhedron $P$ is called \textit{flexible} if its spatial shape can be changed
continuously (i.e., without jumps) only by changing its dihedral angles, i.e., if $P$ 
can be included in a continuous family of polyhedra $\{P_t\}_{t\in [\alpha, \beta)}$ 
such that $P=P_{\alpha}$ and, for any $t\in (\alpha, \beta)$, $P_{\alpha}$ and $P_t$ 
are combinatorially equivalent, their corresponding faces are congruent, while
$P_{\alpha}$ and $P_t$ themselves are not congruent. 
We call such a family $\{P_t\}_{t\in [\alpha, \beta)}$ a \textit{flex} of~$P$, 
and we call $t$ the \textit{parameter} of the flex.

The first examples of flexible polyhedra without boundary in 
$\mathbb{R}^3$ were constructed by R. Bricard in 1897 in \cite{Br97}.
All of them have self-intersections and are combinatorially equivalent to a 
regular octahedron (which is why they are called flexible octahedra). 
Moreover, in \cite{Br97} a classification of all flexible octahedra in 
$\mathbb{R}^3$ is given.
Nowadays, flexible octahedra are commonly called Bricard octahedra. 
One can read more about them, e.g., in \cite{Br97, Le67, Al10, GGLS21, Mi23}.
 
The theory of flexible polyhedra began to flourish after R. Connelly 
constructed a flexible polyhedron in $\mathbb{R}^3$ in 1977 in 
\cite{Co77} that is homeomorphic to a sphere and has no self-intersections.
As the theory of flexible polyhedra developed, they turned out to have many 
remarkable properties; for example, for every orientable flexible polyhedron 
without boundary in $\mathbb{R}^3$, its integral mean curvature \cite{Al85}, 
volume \cite{Sa96, CSW97, Sa98, Ga18}, and every Dehn invariant \cite{GI18} 
are preserved diring the flex.

At the same time, there are still many interesting open questions in the theory 
of flexible polyhedra.
These include the following problem posed by I.Kh. Sabitov:

\begin{probl}\label{probl1}
Is there a flexible polyhedron in $\mathbb{R}^3$, without boundary and without 
self-intersections, for which all dihedral angles change during the flex?
\end{probl} 

Its formulation can be found, e.g., in \cite[p.~182]{Sh15} and 
\cite[Problem 1.3]{Ga18}.

The absence of self-intersections is fundamentally important in Problem~\ref{probl1}.
Indeed, it is easy to see (see Lemma~\ref{lemma1} below) that every dihedral 
angle of any Bricard octahedron changes during the flex.
But all Bricard octahedra have self-intersections.

Theorem~\ref{thrm1} gives a positive answer to Problem~\ref{probl1} and is the 
main result of this article.

\begin{thrm}\label{thrm1}
In Euclidean 3-space $\mathbb{R}^3$, there exists a polyhedron $\mathscr{P}$ 
with the following properties:
\begin{enumerate}
\renewcommand{\theenumi}{\roman{enumi}}
\item $\mathscr{P}$ has only triangular faces, has no self-intersections and 
      is homeomorphic to the sphere $\mathbb{S}^2$, 
\item there exists a flex of $\mathscr{P}$ such that none of its 
      dihedral angles remains constant.
\end{enumerate}
\end{thrm}

As we were finishing this article, A.A. Gaifullin sent us the graduate work
of O.A. Zaslavsky \cite{Za19}, defended under his supervision in 2019 at the 
Chair of Higher Geometry and Topology of the Mathematics and Mechanics Department
at Lomonosov Moscow State University.
It turned out that the polyhedron $\mathscr{P}$ constructed by us from the 
formulation of Theorem~\ref{thrm1} had already been constructed in \cite{Za19};
moreover, it was constructed specifically to answer Problem~\ref{probl1}.
It also turned out that O.A. Zaslavsky argues for the absence of self-intersections 
in $\mathscr{P}$ differently than we do, and does not substantiate the fact that all
dihedral angles change.
The results of O.A. Zaslavsky's graduate work \cite{Za19} have not yet been published
in a refereed journal.
Therefore, we decided to publish our proof of Theorem~\ref{thrm1}.

The plan of the article is as follows. 
In~\S~\ref{sec2} we clarify the terminology and recall 
the construction and properties of the Bricard's octahedron of the first type.
In~\S~\ref{sec3} we recall the information we need about the construction and 
properties of the Steffen's flexible polyhedron.
In~\S~\ref{sec4} we modify the Steffen's polyhedron so that it can be used to 
construct the polyhedron $\mathscr{P}$ of Theorem~\ref{thrm1}.
In~\S~\ref{sec5} we propose an algorithm for solving the problem of whether a 
given polyhedron has self-intersections; by applying its computer implementation 
to the modified Steffen's polyhedron we verify that it has no self-intersections.
In~\S~\ref{sec6} we explicitly construct a polyhedron $\mathscr{P}$ of 
Theorem~\ref{thrm1} and, using the algorithm of \S~\ref{sec5} and the
\textit{Mathematica} software system \cite{Wo99}, verify that it has no 
self-intersections.
In~\S~\ref{sec7} we construct a special 1-parameter flex of $\mathscr{P}$ in which 
none of dihedral angles of $\mathscr{P}$ remains constant; in this section, we again 
make use of \textit{Mathematica}.
Finally, in~\S~\ref{sec8} we gather the results of previous sections to obtain a proof
of Theorem~\ref{thrm1} and formulate open problems related to Problem~\ref{probl1}.

\section{Clarification of terminology and 
Bricard octahedron of type 1}\label{sec2} 
 
Since the article will feature polyhedra with self-intersections, it makes sense 
to clarify the terminology associated with this now.
 
Let $M$ be an abstract two-dimensional manifold glued from a finite number of 
Euclidean triangles $\Delta_k$, $k=1,\dots, n$.
It is possible that $M$ has a non-empty boundary. 
Let $f: M\to\mathbb{R}^3$ be a continuous mapping whose restriction to each 
$\Delta_k$ is a linear isometric embedding.
Then we call $f(M)$ a \textit{polyhedron} or \textit{polyhedral surface} in 
$\mathbb{R}^3$.
If $\delta\subset M$ coincides with one of $\Delta_k$'s, or with its side or vertex, 
then we call $f(\delta)$ a \textit{face}, \textit{edge}, or, respectively, a
\textit{vertex} of the polyhedron $f(M)$.  
\textit{A diagonal} of a polyhedron is a straight line segment that connects two of 
its vertices, but is not an edge.
A diagonal is called \textit{small} if its ends belong to adjacent faces.

We say that $f(M)$ \textit{has no self-intersections}, if $f: M\to\mathbb{R}^3$ 
is injective.
We call $x\in f(M)$ a \textit{point of self-intersection} of $f(M)$, 
if its complete preimage $f^{-1}(x)\subset M$ consists of more than one point.

\textit{An octahedron} is a polyhedron $f(M)$ (convex or non-convex, self-intersecting 
or not) such that the abstract manifold $M$ is combinatorially equivalent to 
the regular convex octahedron shown in Fig.~\ref{f1}.
Unless otherwise stated, we use notations of Fig.~\ref{f1} for the vertices of an
arbitrary octahedron.
\begin{figure}[h]
\begin{center}
\includegraphics[width=0.3\textwidth]{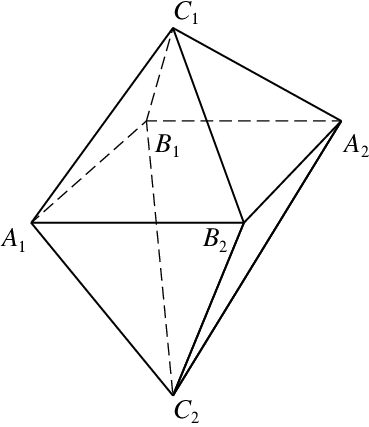}
\end{center}
\caption{Designations of the vertices of a regular tetrahedron.}\label{f1}
\end{figure}

For our purposes it will be sufficient to recall the construction of only 
the Bricard octahedron of type 1.
For the constructions of Bricard octahedra of types 2 and 3, we refer the reader to 
the articles \cite{Br97, Le67, Al10, GGLS21, Mi23} and the literature cited therein.

Consider a disk-homeomorphic polyhedron $\mathscr{D}$ in $\mathbb{R}^3$ 
consisting of four triangles $A_1B_1C_2$, $B_1A_2C_2$, $A_2B_2C_2$, and $B_2A_1C_2$.
Its boundary is the closed spatial broken line $A_1B_1A_2B_2$, from which we require 
that the lengths of its opposite sides are equal to each other, i.e., we require that 
$|A_1B_1|=|A_2B_2|$ and $|B_1A_2|=|B_2A_1|$ (see the left part of Fig.~\ref{f2}).
\begin{figure}[h]
\begin{center}
\includegraphics[width=0.9\textwidth]{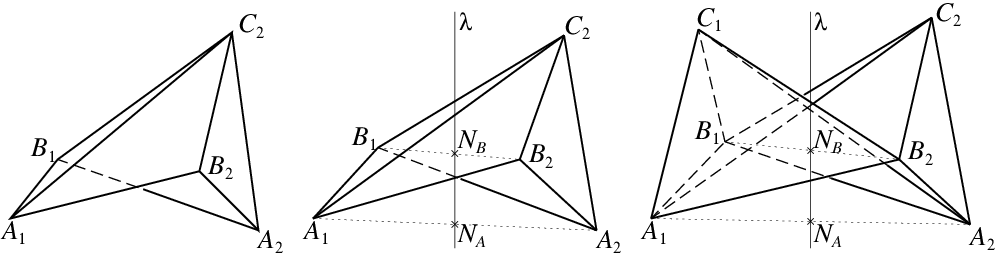}
\end{center}
\caption{Polyhedron $\mathscr{D}$ and Bricard octahedron of type 1 $\mathscr{B}$.}\label{f2}
\end{figure}

Let $N_A$ denote the midpoint of the segment $A_1A_2$ and $N_B$ denote the 
midpoint of the segment $B_1B_2$ (see the central part of Fig.~\ref{f2}).
If $N_A\neq N_B$, then we denote by $\lambda$ the line passing through $N_A$ and 
$N_B$ (see the central part of Fig.~\ref{f2}).
If $N_A=N_B$, then we denote by $\lambda$ the line passing through the point 
$N_A=N_B$ perpendicular to the plane containing the segments $A_1A_2$ and $B_1B_2$.

First of all, we note that the quadrilateral $A_1B_1A_2B_2$ is mapped onto itself 
under the rotation of the entire space $\mathbb{R}^3$ around the line $\lambda$ 
by $180^{\circ}$.
This is obvious if $N_A=N_B$.
If $N_A\neq N_B$, then the equality of triangles $A_1B_1B_2$ and $A_2B_1B_2$ 
implies $|A_1N_B|=|A_2N_B|$ (see the central part of Fig.~\ref{f2}).
Therefore, the triangle $A_1A_2N_B$ is isosceles.
This means that its median $N_AN_B$ is also its height, i.e. the line 
$\lambda = N_AN_B$ is perpendicular to the line $A_1A_2$.
Therefore, when the entire space is rotated around $\lambda$ by $180^{\circ}$, 
the points $A_1$ and $A_2$ exchange places. 
Similarly, starting from the triangles $A_1A_2B_1$ and $A_1A_2B_2$ we conclude 
that the line $\lambda = N_AN_B$ is perpendicular to the line $B_1B_2$, and 
therefore the points $B_1$ and $B_2$ also exchange places when rotating around 
$\lambda$ by $180^{\circ}$.
Thus, we have proved that under such a rotation the quadrilateral $A_1B_1A_2B_2$ 
is mapped into itself.

Now we glue the polyhedron $\mathscr{D}$ and its image under the rotation of 
$\mathbb{R}^3$ around $\lambda$ by $180^{\circ}$ along the sides of the 
quadrilateral $A_1B_1A_2B_2$ (see the right part of Fig.~\ref{f2}).
The resulting polyhedron is a Bricard octahedron of type 1.
We denote it by $\mathscr{B}$.
The image of the point $C_2$ under the rotation of $\mathbb{R}^3$ around $\lambda$ 
by $180^{\circ}$ is denoted by $C_1$.

It follows directly from the above construction that $\mathscr{B}$ is combinatorially
equivalent to an octahedron, has self-intersections and allows a one-parameter flex
(recall that, by definition, the Euclidean distance between at least some two 
vertices of the polyhedron is not preserved during the flex).

To prove Theorem~\ref{thrm1} we need one well-known property of the Bricard octahedra 
(of all types, not just type 1). 
We formulate and prove it in the form of Lemma~\ref{lemma1}.

\begin{lemma}\label{lemma1}
Let a Bricard octahedron be located in $\mathbb{R}^3$ in such a way that none of its 
dihedral angles is equal to 0 or $\pi$.
Then the value of each of its dihedral angles does not remain constant during the flex.
\end{lemma}

\begin{proof}
Assume the converse, i.e., assume that one of the dihedral angles remains constant.
Without loss of generality, we can assume that this is the angle at the edge $A_1B_1$
(recall that we use the same notation for the vertices as in Fig.~\ref{f1}).
Then the length of the diagonal $C_1C_2$ is constant during the flex.
Hence the lengths of all edges of the tetrahedra $A_1B_1C_1C_2$ and $A_1B_2C_1C_2$ 
remain constant.
Thus all dihedral angles of these tetrahedra remain constant.
Therefore, the dihedral angle at the edge $A_1B_2$ of the Bricard octahedron is 
constant (indeed, it is equal to the dihedral angle of $A_1B_2C_1C_2$ at~$A_1B_2$).

Observe that the triangle $A_1C_1C_2$ is non-degenerate in the sense that its 
vertices do not lie on the same line (otherwise the dihedral angle at the edge 
$A_1B_1$ of the Bricard octahedron would be equal to 0 or $\pi$, which contradicts 
the conditions of Lemma~\ref{lemma1}).
Consequently, the tetrahedra $A_1B_1C_1C_2$ and $A_1B_2C_1C_2$ are adjacent to 
each other along the nondegenerate face $A_1C_1C_2$.
This means that the dihedral angles at the edges $A_1C_1$ and $A_1C_2$ of the 
Bricard octahedron are either the sum or the difference of the dihedral angles at 
the same edges in the tetrahedra $A_1B_1C_1C_2$ and $A_1B_2C_1C_2$, and therefore 
are also constant during the flex.

Thus, we have proven that if the dihedral angle at the edge $A_1B_1$ of the 
Bricard octahedron is constant during the flex, then the dihedral angles at its 
edges $A_1B_2$, $A_1C_1$ and $A_1C_2$ are also constant.
In other words, we have proven that if, under the conditions of Lemma~\ref{lemma1}, 
a vertex is incident to an edge, the dihedral angle at which remains constant, 
then the dihedral angles of all edges incident to this vertex remain constant.
It immediately follows that the dihedral angles remain constant for all edges of the
Bricard octahedron, and therefore the lengths of all its diagonals remain constant.
The latter, however, contradicts our definition of the flex.
This contradiction proves Lemma~\ref{lemma1}.
\end{proof}

\section{Steffen polyhedron $\mathscr{S}$}\label{sec3} 
 
As is known, the flexible Steffen polyhedron, denote it by $\mathscr{S}$, is 
obtained by gluing together a certain tetrahedron, denote it by $\mathscr{T}$, 
and two copies of the same Bricard octahedron of type 1, denote it by $ \mathscr{B}$.
Gluing is carried out along congruent faces.
In \S~3 we resemble this well-known construction. 
It can also be found, for example, in \cite{Al10} and the literature mentioned there.
Additional insight into the flex of $\mathscr{S}$ can be obtained via the 
computer animation \cite{Mc07}.

Let the tetrahedron $\mathscr{T}=T_1T_2T_3T_4$ (see the left side of Fig.~\ref{f3}) 
have the following edge lengths: $|T_1T_4|=17$, 
$|T_1T_2|=|T_1T_3|=|T_2T_4| =|T_3T_4|=12$, and $|T_2T_3|=11$.
$\mathscr{T}$~will not change its spatial shape during the flex of~$\mathscr{S}$.
Therefore, throughout \S\S~\ref{sec3}--\ref{sec7} we assume that the points 
$T_j$ ($j=1,\dots, 4$) occupy a fixed position in space.
\begin{figure}[h]
\begin{center}
\includegraphics[width=0.75\textwidth]{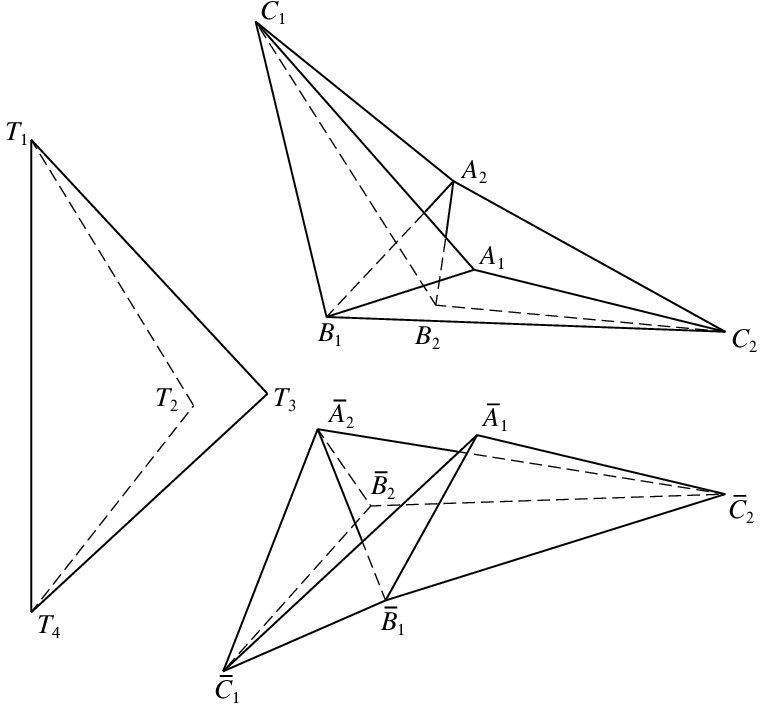}
\end{center}
\caption{Gluing the Steffen polyhedron $\mathscr{S}$.}\label{f3}
\end{figure}

Let $\mathscr{B}=A_1A_2B_1B_2C_1C_2$ be a Bricard octahedron of type 1 
(see the upper right part of Fig.~\ref{f3}), whose edges have the following lengths:
$|A_1C_1|=|B_1C_2|=|A_2C_2|=|B_2C_1|=12$, 
$|A_2C_1|=|B_1C_1|=|A_1C_2|=|B_2C_2|=10$, 
$|A_1B_1|=|A_2B_2|=5$, and $|A_1B_2|=|A_2B_1|=11$.

Let us move $\mathscr{B}$ in $\mathbb{R}^3$ by means of an orientation-preserving 
motion so that the following pairs of points coincide: $T_1$ and $C_1$, 
$T_2$ and $B_2$, and $ T_3$ and $A_1$.
Note that here we do not care whether $\mathscr{T}$ intersects $\mathscr{B}$, after 
the above mentioned motion, somewhere outside the faces $T_1T_2T_3$ and $A_1B_2C_1$. 
In this case, we say that we have glued $\mathscr{T}$ and $\mathscr{B}$ along 
$T_1T_2T_3$ and $A_1B_2C_1$. 
Since we want the result of gluing two polyhedra to be a polyhedron again, here 
and below we mean that both faces along which the gluing was made are deleted.
We denote the result of gluing $\mathscr{T}$ and $\mathscr{B}$ 
by $\mathscr{T}\sqcup\mathscr{B}$.

Now consider another copy
$\overline{\mathscr{B}}=\overline{A}_1\overline{A}_2
\overline{B}_1\overline{B}_2\overline{C}_1\overline{C}_2$
of the Bricard octahedron of type 1 which is obtained from $\mathscr{B}$ by an
orientation-preserving motion of~$\mathbb{R}^3$.
Thus, each edge of $\overline{\mathscr{B}}$ has the same length as the 
corresponding edge of $\mathscr{B}$.
For example, $|\overline{A}_1\overline{C}_1|=|A_1C_1|=12$.
 
Glue $\mathscr{T}\sqcup\mathscr{B}$ and $\overline{\mathscr{B}}$ along the faces
$T_2T_3T_4$ and $\overline{A}_1\overline{B}_2\overline{C }_1$, i.e., move 
$\overline{\mathscr{B}}$ in $\mathbb{R}^3$ by means of an orientation-preserving motion
so that the following pairs of vertices coincide: $T_2$ and $\overline{B}_2$, 
$T_3 $ and $\overline{A}_1$, as well as $T_4$ and $\overline{C}_1$ 
(see the lower right part of Fig.~\ref{f3}).
The polyhedron obtained as a result of such gluing is denoted by 
$(\mathscr{T}\sqcup\mathscr{B})\sqcup\overline{\mathscr{B}}$. 
Obviously, it is homeomorphic to the disk.

Since $\mathscr{B}$ and $\overline{\mathscr{B}}$ are Bricard octahedra, 
the positions of the vertices $C_2$ and $\overline{C}_2$ of 
$(\mathscr{T}\sqcup \mathscr{B})\sqcup\overline{\mathscr{B}}$ are not determined 
uniquely by the above gluing $\mathscr{T}$ and $\mathscr{B}$ and the subsequent 
gluing $\mathscr{T} \sqcup\mathscr{B}$ and $\overline{\mathscr{B}}$.
In fact, after such gluing, $C_2$ and $\overline{C}_2$ can lie at any points 
of the circle $\gamma$ defined by the following conditions: $\gamma$ is located 
in a plane perpendicular to the segment $T_2T_3$; the center of $\gamma$ is
the midpoint of $T_2T_3$; the radius of $\gamma$ is equal to 
$\sqrt{|A_1C_2|^2-|A_1B_2|^2/4}= \sqrt{10^2-11^2/4}=(3/2)\sqrt{31} \approx 8.35$.
Thus, by choosing some position of $C_2$ on $\gamma$, we can (without changing 
the position of $\mathscr{T}$ in space) bend the Bricard octahedron 
$\overline{\mathscr{B}}$ so that $\overline{C}_2$ coincides with $C_2$.
In this position we glue the polyhedron
$(\mathscr{T}\sqcup\mathscr{B})\sqcup\overline{\mathscr{B}}$
with itself along the faces $A_1B_2C_2$ and 
$\overline{A}_1\overline{B}_2\overline{C}_2$.
Denote the resulting polyhedron by
$((\mathscr{T}\sqcup\mathscr{B})\sqcup\overline{\mathscr{B}})_{\overline{C}_2=C_2}$.
Observe that, in order 
$((\mathscr{T}\sqcup\mathscr{B})\sqcup\overline{\mathscr{B}})_{ \overline{C}_2=C_2}$ 
is a polyhedron, in the course of gluing we remove not only the internal points of 
the glued faces $A_1B_2C_2$ and $\overline{A}_1\overline{B}_2\overline{C}_2$ 
(this was already explained above), but also the internal points of the segment 
$A_1B_2=\overline{A}_1\overline{B}_2$ (which, therefore, is not an edge of 
$((\mathscr{T}\sqcup\mathscr{B})\sqcup\overline{\mathscr{B}})_{\overline{C}_2=C_2})$).
 
It follows directly from the construction that 
$((\mathscr{T}\sqcup\mathscr{B})\sqcup \overline{\mathscr{B}})_{\overline{C}_2=C_2}$
form a continuous family of combinatorially equivalent polyhedra with congruent
corresponding faces, and that the polyhedra corresponding to different positions 
of the point $\overline{C}_2=C_2$ on $\gamma$ are not congruent to each other. 
Thus, this family is a flex of any of the polyhedra included in it, or, 
equivalently, any polyhedron in this family is flexible.
Let us select one polyhedron in this family as follows.
 
Let us denote by $X$ the midpoint of the segment $T_1T_4$, by $Y$ the midpoint of 
the segment $T_2T_3$, and by $Z$ the intersection point of the ray 
$\stackrel{\longrightarrow}{XY}$ with the circle $\gamma$.
If the point $\overline{C}_2=C_2$ coincides with $Z$, then we call 
$((\mathscr{T}\sqcup\mathscr{B})\sqcup \overline{\mathscr{B}})_{\overline{C}_2=C_2}$  
the \textit{Steffen polyhedron} and denote it by $\mathscr{S}$.
We consider it a generally known fact that $\mathscr{S}$ has no self-intersections.
Note, however, that the technology we develope in \S~\ref{sec5} for checking the 
absence of self-intersections in a polyhedron allows us to strictly prove that 
$\mathscr{S}$, as well as any polyhedron 
$((\mathscr{ T}\sqcup\mathscr{B})\sqcup \overline{\mathscr{B}})_{\overline{C}_2=C_2}$,
sufficiently close to $\mathscr{S}$, indeed has no self-intersections. 
Although this will not be needed in this article, we can even specify that 
$((\mathscr{T}\sqcup\mathscr{B})\sqcup\overline{\mathscr{B}})_{\overline{C}_2=C_2}$ 
has no self-intersections provided the angle between the segments $XZ$ and $XC_2$ 
is less than $7.5^{\circ}$.

Using the following agreements, we assign a permanent designation to each vertex $V$
of $\mathscr{S}$ and always use it in \S\S~\ref{sec5}--\ref{sec7}:

$\bullet$ if, before gluing, $V$ belonged to only one of the polyhedra 
$\mathscr{T}$, $\mathscr{B}$ or $\overline{\mathscr{B}}$, then we reserve for it 
the designation that it had on that polyhedron;

$\bullet$ if $V$ is the result of gluing a vertex $W$ of $\mathscr{T}$ with 
a vertex of $\mathscr{B}$ and/or $\overline{\mathscr{B}}$, then we assign the 
vertex $V\in\mathscr{S}$ the designation that $W$ had on $\mathscr{T}$;

$\bullet$ if $V$ is the result of gluing a vertex $W$ of $\mathscr{B}$ with 
a vertex of $\overline{\mathscr{B}}$, then we assign the vertex $V\in \mathscr{S}$ 
the designation that $W$ had on $\mathscr{B}$.

For example, in the process of construction of $\mathscr{S}$, we 
first glued the vertices $T_3\in\mathscr{T}$ and $A_1\in \mathscr{B}$, and 
then glued the resulting point to the vertex 
$\overline{ A}_1\in \overline{\mathscr{B}}$.
In accordance with what has been said, we denote the resulting vertex of $\mathscr{S}$ 
by $T_3$.
Another example: at the last stage of construction of $\mathscr{S}$, we glued together 
the vertices $C_2\in \mathscr{B}$ and $\overline{C}_2\in \overline{\mathscr{B}}$.
Hence, we denote the resulting vertex of $\mathscr{S}$ by $C_2$.

\section{Construction of a modified Steffen polyhedron $\mathscr{M}$}\label{sec4} 

By construction, the edge $T_1T_4$ of the Steffen polyhedron $\mathscr{S}$ has 
length 17.
Direct calculations show that the internal dihedral angle at this edge is equal 
to $\arccos (45/287)\approx 80^{\circ}59'$.
Let us change the length of $T_1T_4$ so that it becomes equal to
$\sqrt{334}\approx 18.28$. 
We do not change neither the lengths of the remaining edges of $\mathscr{S}$, 
nor the designations of its vertices.
The resulting polyhedron is called a \textit{modified Steffen polyhedron} and 
is denoted by $\mathscr{M}$.
Direct calculation shows that the dihedral angle of $\mathscr{M}$ at the edge 
$T_1T_4$ is equal to $90^{\circ}$.
It is clear directly from the construction of~$\mathscr{M}$ that it is 
combinatorially equivalent to a sphere, has only triangular faces, and is flexible.
The development of~$\mathscr{M}$ is shown in Fig.~\ref{f4}.
The reader can scan it, print it on a larger scale on thick paper and glue 
together a model of~$\mathscr{M}$.
This will simplify the understanding of our further constructions.
We especially emphasize that by construction the vertex $C_2\in \mathscr{M}$ 
coincides in $\mathbb{R}^3$ with the point $Z$ constructed in \S~\ref{sec3}.
\begin{figure}[h]
\begin{center}
\includegraphics[width=0.85\textwidth]{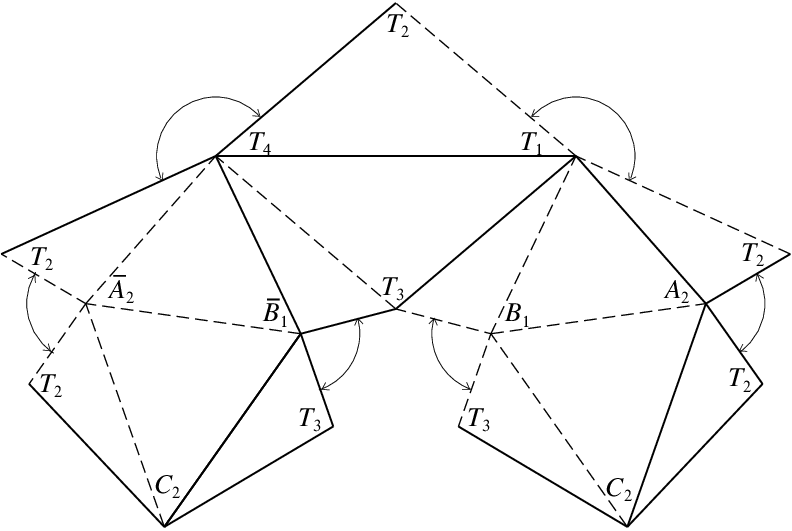}
\end{center}
\caption{The development of a modified Steffen polyhedron~$\mathscr{M}$.
It must be bent so that the internal dihedral angle of 
$\mathscr{M}$ is less than $180^{\circ}$ at the edges drawn with solid lines 
and is greater than $180^{\circ}$ at the edges drawn with dotted lines.
The sides of the development with identical vertices need to be glued in pairs; 
more precisely, one must glue in pairs the sides connected to each other by 
circular arcs, two sides $T_2C_2$, and two sides $T_3C_2$.
}\label{f4}
\end{figure}

We associate with $\mathscr{M}$ the following Cartesian coordinate system 
in $\mathbb{R}^3$.
The origin is at the point $X$, which was defined in \S~\ref{sec3} as the midpoint of
the segment $T_1T_4$.
We choose the $x$-axis so that it passes through the point $T_3$ and~$T_3$ has 
a positive $x$-coordinate.
We choose the $y$-axis so that it passes through the point $T_2$ and~$T_2$ has 
a positive $y$-coordinate.
And we choose the $z$-axis so that it passes through the point $T_1$ and~$T_1$ has 
a positive $z$-coordinate.
We use this and only this coordinate system in \S\S~\ref{sec4}--\ref{sec7}.

Vertex $A_2$ is connected by edges of~$\mathscr{M}$ with its three vertices 
$T_1$, $T_2$ and $C_2$, whose coordinates are known to us directly from the 
constructions of $\mathscr{M}$ and of the coordinate system. 
They are shown in Table~\ref{tab1}.
Using \textit{Mathematica} we solve symbolically the system
\begin{equation}
\begin{cases}
|A_2T_1|^2 & =10^2,\\
|A_2T_2|^2 & =5^2,\\
|A_2C_2|^2 & =12^2
\end{cases}\label{eq1}
\end{equation}
of three algebraic equations of the second degree with respect to the coordinates 
of~$A_2$, and we obtain two exact (i.e., expressed in radicals) solutions.
It is obvious that the two points in $\mathbb{R}^3$ corresponding to these 
solutions are symmetrical to each other with respect to the plane passing through 
the vertices $T_1$, $T_2$ and $C_2$.
Considering the model of $\mathscr{M}$, we conclude that vertex $A_2$ corresponds 
to the solution whose $z$-coordinate is the largest (since only in this case the 
inner dihedral angle of~$\mathscr{M}$ at edge $A_2T_2$ is less than $180^{\circ}$;
the latter agrees with the development shown in Fig.~\ref{f4}).
We put the exact values of the coordinates of vertex $A_2$ found in this way 
on Table~\ref{tab1}.
The approximate values of the coordinates of vertices shown in Table~\ref{tab1}
are not used in our reasoning and are presented solely to make it easier for the 
reader to imagine the spatial form of~$\mathscr{M}$.
For example, the fact that the approximate value of the $x$-coordinate of point $A_2$
is negative suggests that $A_2$ (and therefore ~$\mathscr{M}$) is not contained
in the quarter of the space consisting of the points with positive $x$- and 
$y$-coordinates.
\begin{table}[h]
\begin{tabular}{llll}
\hline
   &\hphantom{A}{\small$x$-coordinate} 
   &\hphantom{A}{\small$y$-coordinate} 
   &\hphantom{A}{\small$z$-coordinate}\\
\hline
$T_1$   & {\small\hphantom{$-$}$0$} & {\small\hphantom{$-$}$0$} 
        & {\small\hphantom{$-$}$\frac{\sqrt{167}}{\sqrt{2}}\approx 9.13 
        $} \\
$T_2$   & {\small\hphantom{$-$}$0$} 
        & {\small\hphantom{$-$}$\frac{11}{\sqrt{2}}\approx 7.77 
        $} 
        & {\small\hphantom{$-$}$0$} \\
$T_3$   & {\small\hphantom{$-$}$\frac{11}{\sqrt{2}}\approx 7.77 
        $} 
        & {\small\hphantom{$-$}$0$} & {\small\hphantom{$-$}$0$} \\
$T_4$   & {\small\hphantom{$-$}$0$} & {\small\hphantom{$-$}$0$} 
        & {\small$-\frac{\sqrt{167}}{\sqrt{2}}\approx -9.13 
        $} \\ 
$C_2$   & {\small\hphantom{$-$}$\frac{11+3\sqrt{31}}{2\sqrt{2}}\approx 9.79 
        $} 
        & {\small\hphantom{$-$}$\frac{11+3\sqrt{31}}{2\sqrt{2}}\approx 9.79 
        $} 
        & {\small\hphantom{$-$}$0$} \\
$A_2$   & {\small\hphantom{$-$}$\frac{\omega_1\sqrt{2}-
        2(200-33\sqrt{31})\sqrt{\rho}}{1230304047998}\approx -1.19$} 
        & {\small\hphantom{$-$}$\frac{\omega_2\sqrt{2}+2\sqrt{\rho}}{15573468962}
        \approx 8.89 
        $} 
        & {\small\hphantom{$-$}$\frac{167\omega_3\sqrt{2}+
        22\sqrt{\rho}}{15573468962\sqrt{167}}\approx 4.72 
        $} \\
$B_1$   & {\small\hphantom{$-$}$\frac{\omega_2\sqrt{2}-
        2\sqrt{\rho}}{15573468962}\approx 2.79 
        $} 
& {\small\hphantom{$-$}$\frac{\omega_1\sqrt{2}+
        2(200-33\sqrt{31})\sqrt{\rho}}{1230304047998}\approx 0.05
        $}
& {\small\hphantom{$-$}$\frac{167\omega_3\sqrt{2}-
        22\sqrt{\rho}}{15573468962\sqrt{167}}\approx -0.46 
        $} \\
$\overline{A}_2$    & {\small\hphantom{$-$}$\frac{\omega_1\sqrt{2}+
        2(200-33\sqrt{31})\sqrt{\rho}}{1230304047998}\approx 0.05
        $} 
        & {\small\hphantom{$-$}$\frac{\omega_2\sqrt{2}-
        2\sqrt{\rho}}{15573468962}\approx 2.79 
        $}  
& {\small$-\frac{167\omega_3\sqrt{2}-
        22\sqrt{\rho}}{15573468962\sqrt{167}}\approx 0.46 
        $} \\
$\overline{B}_1$    & {\small\hphantom{$-$}$\frac{\omega_2\sqrt{2}+
        2\sqrt{\rho}}{15573468962}\approx 8.89 
        $} 
        & {\small\hphantom{$-$}$\frac{\omega_1\sqrt{2}-
        2(200-33\sqrt{31})\sqrt{\rho}}{1230304047998}\approx -1.19
        $} 
        & {\small$-\frac{167\omega_3\sqrt{2}+
        22\sqrt{\rho}}{15573468962\sqrt{167}}\approx -4.72 
        $} \\
\hline
\end{tabular}
\caption{Exact and approximate values of the coordinates of the vertices 
of~$\mathscr{M}$.
All decimal places in approximate values are correct, i.e., are written without 
taking into account rounding rules. 
To shorten expressions for exact values, the following notation is used:
$\rho=167(1712315512948039256+297671463726717927\sqrt{31})$,
$\omega_1=237(670333576-497644539\sqrt{31})$,
$\omega_2=3(26431711823-892912093\sqrt{31})$, 
$\omega_3=2798420941-176443707\sqrt{31}$.}\label{tab1}
\end{table}
 
Similarly, we find the exact coordinates of the vertex $B_1$.
To do this, we use the fact that $B_1$ is connected by edges of $\mathscr{M}$ 
with its three vertices $T_1$, $T_3$ and $C_2$. 
But this time, from two solutions to the corresponding system of three algebraic
equations of the second degree, similar to system (\ref{eq1}), we choose the one 
with the smallest $z$-coordinate.
We put the exact values of the coordinates of $B_1$ found in this way 
on Table~\ref{tab1}.

We check the correctness of the above described method for recognizing the 
coordinates of $A_2$ among solutions of system~(\ref{eq1}) and the coordinates 
of $B_1$ among solutions of a similar system by calculating the length of~
$A_2B_1$ through the coordinates of its ends given in Table~\ref {tab1}.
Symbolic calculations in \textit{Mathematica} show that this length is indeed 
equal to 11. 

The coordinates of the vertices $\overline{A}_2$ and $\overline{B}_1$ can be 
found in the same way as we found the coordinates of the vertices $A_2$ and $B_1$.
However, these calculations can be avoided using the following lemma.

\begin{lemma}\label{lemma2}
The modified Steffen polyhedron $\mathscr{M}$ transforms into itself under the action of
rotation of the entire space $\mathbb{R}^3$ by $180^{\circ}$ around the line $XY$ 
passing through the points $X$ and $Y$, constructed in \S~\ref{sec3}, i.e. under the
rotation of $\mathbb{R}^3$ defined by the matrix
\begin{equation*}
L=\begin{pmatrix} 
0 & 1 & 0 \\ 
1 & 0 & 0 \\
0 & 0 & -1  
\end{pmatrix}.
\end{equation*} 
\end{lemma}

\begin{proof}
Using Table~\ref{tab1}, we directly verify that
$L(T_1)=T_4$, $L(T_2)=T_3$, $L(T_3)=T_2$, $L(T_4)=T_1$, and $L(C_2)=C_2$.
Therefore, the distances from the point $L(A_2)$ to the vertices $T_3$, $T_4$, 
and $C_2$ are the same as the distances from the point $\overline{B}_1$ to 
the specified vertices.
Indeed,
$|L(A_2)T_3|=|L(A_2)L(T_2)|=|A_2T_2|=|A_2B_2|=5$ and
$|\overline{B}_1T_3|=|\overline{B}_1\overline{A}_1|=5$;
$|L(A_2)T_4|=|L(A_2)L(T_1)|=|A_2T_1|=|A_2C_1|=10$ and 
$|\overline{B}_1T_4|=|\overline{B}_1\overline{C}_1|=10$;
$|L(A_2)C_2|=|L(A_2)L(C_2)|=|A_2C_2|=12$ and 
$|\overline{B}_1C_2|=12$.
In addition, both points $L(A_2)$ and $\overline{B}_1$ lie on the same side 
of the plane passing through the vertices $T_3$, $T_4$, and $C_2$.
Indeed, according to Fig.~\ref{f4}, the inner dihedral angle of~$\mathscr{M}$ at 
both edges $\overline{B}_1T_3=\overline{B}_1\overline{A}_1$ and $A_2T_2=A_2B_2$
(and hence the edge $L(A_2)L(T_2)=L(A_2)L(B_2 )$) is less than $180^{\circ}$.
Hence, $L(A_2)=\overline{B}_1$.

Reasoning in a similar way, we can prove that the distances from $L(B_1)$ to 
the vertices $T_2=L(T_3)$, $T_4=L(T_1)$, $C_2=L(C_2)$ are the same as distances 
from $\overline{A}_2$ to the specified vertices.
In addition, by analogy with what was said above, it can be proven that both 
points $L(B_1)$ and $\overline{A}_2$ lie on the same side of the plane passing 
through $T_2$, $T_4$, and $C_2 $.
Hence, $L(B_1)=\overline{A}_2$. 

Thus, we are convinced that the set of vertices of $\mathscr{M}$ (as well as 
the sets of its edges and faces) is mapped onto itself by~$L$.
This means that $\mathscr{M}$ is mapped onto itself by~$L$.
\end{proof}

The proof of Lemma~\ref{lemma2} yields $\overline{B}_1=L(A_2)$ and 
$\overline{A}_2=L(B_1)$. 
Using these relations, we can find coordinates of $\overline{A}_2$ and 
$\overline{B}_1$ without additional calculations and thereby complete 
filling out Table~\ref{tab1}.

So, $\mathscr{M}$ has 9 vertices, 21 edges and 14 faces, and has the 
symmetry described in Lemma~\ref{lemma2}.
The study of flexible and nonrigid frameworks (and therefore polyhedra) with 
symmetry is in itself an interesting and nontrivial problem, see, e.g., 
\cite{CNSW20} and literature cited there.
But we will achieve our goals in the most direct and elementary way, without 
using general results on the effect of symmetry on the rigidity and 
flexibility of frameworks.

\section{Do the polyhedra $\mathscr{M}$ and $\mathscr{S}$ have self-intersections?}\label{sec5} 

To solve Problem~\ref{probl1} we must be able to answer to decide whether 
a given polyhedron has self-intersections.
To simplify and unify our reasoning, in \S~\ref{sec5} we describe an algorithm 
that makes it possible to guarantee the absence of self-intersections for the 
polyhedra $\mathscr{S}$, $\mathscr{M}$ and $\mathscr{P }$ (the latter will be 
built in \S~\ref{sec6}) and, of course, not only for them.

The problem of finding intersections and self-intersections of polyhedral surfaces 
arises in a variety of problems in mathematics and applied mathematics, computer 
graphics and video games.
There are many variations of this problem and dozens of algorithms designed to 
solve them, see, for example, \cite{GJK88}, \cite{GS99}, \cite{Er05}, \cite{JSF10},
\cite{Mo17} and the literature mentioned there.
However, existing algorithms are not suitable for us for the following reasons:

$\bullet$ 
we want to solve the problem of finding self-intersections without using
of floating point arithmetic, i.e., we allow ourselves to use only symbolic 
calculations (which we de facto execute in \textit{Mathematica});

$\bullet$ 
we want the algorithm to be logically as transparent as possible, and therefore 
we do not analyze ``exceptional cases''; instead, the algorithm must notify us of the 
occurrence of each such case (i.e., we are not interested in either complete 
automation of calculations or optimization of their complexity and speed) .

Therefore, we are forced to develop our own algorithm, the description of 
which we proceed to.
 
Recall that, according to~\S\ref{sec2}, a polyhedron is the image $f(M)$ of an 
abstract two-dimensional manifold $M$, glued from a finite number of Euclidean 
triangles $\Delta_k$, $k=1,\dots, n$, under the action of a continuous map 
$f: M\to\mathbb{R}^3$ whose restriction to each triangle $\Delta_k$ is a 
linear isometric embedding.
The point ${\bf x}\in f(M)$ is a self-intersection point of~$f(M)$ if its complete
inverse image $f^{-1}({\bf x})\subset M$ consists from more than one point.

Our algorithm is based on the following considerations.
 
Let ${\bf x}\in f(M)$ be a self-intersection point of $f(M)$ and 
let $u, v\in M$ be two distinct points in $f^{-1}({\ bf x})$.
Denote by $\Delta_{k_i}$, $i=1,2$, two different Euclidean triangles 
from which $M$ is glued, such that $u\in \Delta_{k_1}$ and $v\in \Delta_ {k_2}$.
Recall that in \S~\ref{sec2} we agreed to call the triangles $f(\Delta_{k_1})$ 
and $f(\Delta_{k_2})$ faces of~$f(M)$.
Fig.~\ref{f5} illustraits why, for any mutual arrangement of~$f(\Delta_{k_1})$ and
$f(\Delta_{k_2})$, at least on one of these triangles has a point $\widetilde {\bf x}$
of intersection of a face and an edge of~$f(M)$.
Thus, the problem of whether a polyhedron $f(M)$ has self-intersections is 
reduced to the problem of whether a closed triangle and a closed segment 
in $\mathbb{R}^3$ intersect.
\begin{figure}[h]
\begin{center}
\includegraphics[width=0.75\textwidth]{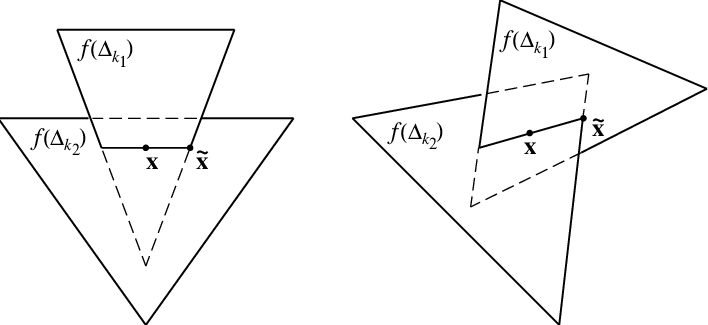}
\end{center}
\caption{Self-intersection point ${\bf x}$ of a polyhedron~$f(M)$ belonging
to faces $f(\Delta_{k_i})\subset f(M)$, $i=1,2$.
The point $\widetilde{\bf x}$ is a point of intersection of a face and an edge 
of~$f(M)$.}\label{f5}
\end{figure}

To solve the last problem, we use the oriented volume of a tetrahedron.
As is known, if ${\bf x}_{k}=(x_{k,1},x_{k,2}, x_{k,3})$, $k=0,1,2,3$, 
are position vectors of the vertices of a tetrahedron in $\mathbb{R}^3$, then 
the oriented volume 
$\operatorname{Vol}({\bf x}_{0}, {\bf x}_{1}, {\bf x}_{2}, {\bf x}_{3})$ 
of this tetrahedron can be calculated using one of the following formulas:
\begin{equation}
\operatorname{Vol}({\bf x}_{0}, {\bf x}_{1}, {\bf x}_{2}, {\bf x}_{3})
={\frac{1}{6}}{\begin{vmatrix}1&x_{0,1}&x_{0,2}&x_{0,3}\\
1&x_{1,1}&x_{1,2}&x_{1,3}\\1&x_{2,1}&x_{2,2}&x_{2,3}\\
1&x_{3,1}&x_{3,2}&x_{3,3}\end{vmatrix}}
={\frac{1}{6}}{\begin{vmatrix}x_{1,1}-x_{0,1}&x_{1,2}-x_{0,2}&x_{1,3}-x_{0,3}\\
x_{2,1}-x_{0,1}&x_{2,2}-x_{0,2}&x_{2,3}-x_{0,3}\\
x_{3,1}-x_{0,1}&x_{3,2}-x_{0,2}&x_{3,3}-x_{0,3}
\end{vmatrix}}.\label{eq2}
\end{equation}	
Recall that the oriented volume of a tetrahedron is equal to its ``usual'' volume 
if the triple of vectors ${\bf x}_{1}-{\bf x}_{0}$, ${\bf x}_{2} -{\bf x}_{0}$, 
${\bf x}_{3}-{\bf x}_{0}$ is right-oriented; is equal to the number opposite to the ``usual'' volume, if the triple of vectors ${\bf x}_{1}-{\bf x}_{0}$, 
${\bf x}_{2}-{\bf x}_{0}$, ${\bf x}_{3}-{\bf x}_{0}$ is left-oriented, and is equal 
to zero if the vectors ${\bf x}_{1}-{ \bf x}_{0}$, ${\bf x}_{2}-{\bf x}_{0}$ and 
${\bf x}_{3}-{\bf x}_{ 0}$ lie in the same plane.

So, let a closed triangle $\Delta\subset\mathbb{R}^3$ be defined by the coordinates 
of its vertices ${\bf y}_{k}=(y_{k,1}, y_{k,2}, y_{ k,3})$, $k=1,2,3$ and a closed
segment $I\subset\mathbb{R}^3$ is given by the coordinates of its vertices
${\bf z}_{j}=(z_{j,1}, z_{j,2}, z_{j,3})$, $j=1,2$.
If both determinants
\begin{equation}\label{eq3}
6\operatorname{Vol}({\bf y}_1, {\bf y}_2, {\bf y}_3, {\bf z}_1), \qquad
6\operatorname{Vol}({\bf y}_1, {\bf y}_2, {\bf y}_3, {\bf z}_2)
\end{equation}
are nonzero and have the same sign, then~${\bf z}_1$ and ${\bf z}_2$ lie srictly 
to one side of the plane containing~$\Delta$ (see the left side of Fig.~\ref{f6}).
\begin{figure}[h]
\begin{center}
\includegraphics[width=0.9\textwidth]{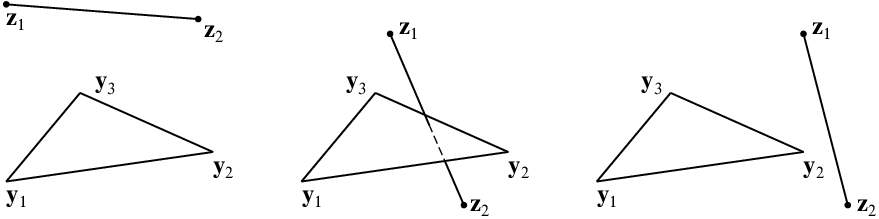}
\end{center}
\caption{Various cases of mutual arrangement of a triangle $\Delta$ with vertices 
${\bf y}_1, {\bf y}_2, {\bf y}_3$ and a segment $I$ with endpoints ${\bf z}_1$, 
$ {\bf z}_2$.}\label{f6}
\end{figure}
Consequently, in this case~$I$ does not intersect the plane containing~$\Delta$, 
and hence $I\cap\Delta=\varnothing$.
If both determinants (\ref{eq3}) are nonzero and have different signs, 
then~${\bf z}_1$ and ${\bf z}_2$ lie to opposite sides of the plane 
containing~$\Delta$ (see the central and right parts of Fig.~\ref{f6}).
In this case, to figure out whether $I$ and $\Delta$ intersect, we need 
to calculate three more determinants:
\begin{equation}\label{eq4}
6\operatorname{Vol}({\bf z}_1, {\bf z}_2, {\bf y}_1, {\bf y}_2),\qquad
6\operatorname{Vol}({\bf z}_1, {\bf z}_2, {\bf y}_2, {\bf y}_3),\qquad
6\operatorname{Vol}({\bf z}_1, {\bf z}_2, {\bf y}_3, {\bf y}_1).
\end{equation}

If all determinants (\ref{eq4}) are non-zero and have the same sign, then
$I\cap\Delta\neq\varnothing$.
This case is schematically shown in the central part of Fig.~\ref{f6}.
To understand why $I\cap\Delta\neq\varnothing$, for each point ${\bf x}$ lying on 
the boundary of the triangle $\Delta$, we denote by $\pi({\bf x}) $ the half-plane
bounded by the line containing~$I$ and passing through~${\bf x}$.
The fact that all three determinants in (\ref{eq4}) have the same sign, implies that, 
when~${\bf x}$ goes around the boundary of~$\Delta$ once, moving all the time in the 
same direction, the half-plane $\pi({\bf x})$ also always rotates in the same direction
and makes a complete turn around the straight line containing~$I$.
This means that the boundary of~$\Delta$ is linked to the staight line containing~$I$.
Taking into account that the ends of~$I$ lie on opposite sides of the plane 
containing~$\Delta$, we conclude that $I\cap\Delta\neq\varnothing$.

Similarly, if all three determinants in (\ref{eq4}) are nonzero, but not all of them 
have the same sign, then after~${\bf x}$ circles the boundary of~$\Delta$ once, moving
all the time in the same direction, the half-plane $\pi({\bf x})$ will return to its
original position, but will not make a complete turn around the straight line 
containing~$I$.
This means that the boundary of~$\Delta$ and the line containing~$I$ are not 
linked curves.
Therefore, in this case $I\cap\Delta=\varnothing$.

This algorithm does not work in ``exceptional'' cases, when at least one of 
the determinants in (\ref{eq3}) and (\ref{eq4}) is zero.
In principle, it can be improved by introducing new functions, somewhat similar 
to the determinants in (\ref{eq3}) and (\ref{eq4}), the use of which will allow 
us to fully understand each of the ``exceptional'' cases.
But we expect that for the polyhedra $\mathscr{M}$, $\mathscr{S}$ and $\mathscr{P}$,
i.e., for the polyhedra of interest to us, there will be no ``exceptional'' cases 
at all or they will be very few so that we can study each of them individually.
Therefore, in order to avoid complicating the algorithm, when encountering an
``exceptional'' case, our algorithm only states that the question of the existence 
of intersections of~$\Delta$ and~$I$ requires additional study.

More specifically, our algorithm can be formulated as follows:

{\bf Step~1.}
[Preliminary work for compiling lists of vertices, edges and faces.]

Enumerate the vertices of the abstract manifold~$M$ in an arbitrary way.
Denote the $r$th vertex of $M$ by $v_r$.
The set~$S$ of all unordered pairs $s=(v_p, v_q)$ of vertices of~$M$ 
such that $v_p$ and $v_q$ are connected by an edge of~$M$ is called
the \textit{list of edges} of~$M$.
The set~$T$ of all unordered triples $t=(v_i, v_j, v_k)$ of vertices
of~$M$ such that $v_i$, $v_j$ and $v_k$ are vertices of some face of~$M$
is called the \textit{list of faces} of~ $M$. 
Fix a linear order $<$ on the Cartesian product $T\times S$ of~$T$ and $S$. 
Generate a list of all vertices ${\bf x}_r=f(v_r)\in\mathbb{R}^3$ of the 
polyhedron~$f(M)$, in which every ${\bf x}_r$ is given by its coordinates 
in~$\mathbb{R}^3$.
Finally, create an empty auxiliary file.

{\bf Step~2.}  
[This step begins exhaustion of all pairs $(t,s)\in T\times S$ and 
verification whether the intersection $f(t)\cap f(s)$ is empty.
The first time we perform Step 2, we set $(t,s)$ equal to the least 
element of the ordered set $(T\times S, <)$.
When Step 2 is repeated, the selection of $(t,s)$ occurs in Step 5. 
In order to bring the notation closer to those previously used in \S~\ref{sec5}, we
assume that the face $t\subset M$ has vertices $u_1$, $u_2$, $u_3$, and the edge 
$s\subset M$ has vertices $w_1$, $w_2$.
We denote the face $f(t)\subset f(M)$ by $\Delta$, denote its vertices by 
${\bf y}_k=f(u_k)\in\mathbb{R}^3$, $k=1,2,3$, and write these vertices in 
coordinates in the form ${\bf y}_k=(y_{k,1}, y_{k,2}, y_{k,3})\in \mathbb{R}^3$.
Similarly, we denote the edge $f(s) \subset f(M)$ by $I$, denote its ends by 
${\bf z}_j=f(w_j)\in\mathbb{R}^3$, $j =1,2$, and write these ends in coordinates 
in the form ${\bf z}_j=(z_{j,1}, z_{j,2}, z_{j,3})\in\mathbb{R }^3$.]

Using the notation just introduced, we can describe Step 2 as follows:

$\bullet$ 
if $t$ and $s$ are not incident to each other, i.e., if $u_k\neq w_j$ for all 
$k=1,2,3$ and all $j=1,2$, then go to {\bf Step 3};

$\bullet$ 
if $t$ and $s$ have exactly one common point, i.e., if the equality $u_k=w_j$ 
holds for only one pair of indices $k$, $j$, then go to {\bf Step 4};

$\bullet$ 
if $t$ and $s$ have more than one common point (i.e., if for any $j=1,2$ there is
$k=1,2,3$ for which the equality $u_k=w_j$ holds or, which is the same, if $s$ 
is a side of $t$), then directly from the definition, it is clear that the pair 
$(s,t)$ does not generate self-intersections of~$f(M)$, and we go to {\bf Step 5.}

{\bf Step~3.} 
[The case when $t$ and $s$ are not incident to each other.]

Calculate two determinants
$6\operatorname{Vol}({\bf y}_1, {\bf y}_2, {\bf y}_3, {\bf z}_1)$ and
$6\operatorname{Vol}({\bf y}_1, {\bf y}_2, {\bf y}_3, {\bf z}_2)$ 
(see formulas (\ref{eq2}), (\ref{eq3})).
Then, 

$\bullet$ if at least one of these two determinants is equal to zero, then write 
the message ``The question whether $f(t)$ and~$f(s)$ do intersect requires 
additional study'' to the auxiliary file, and go to {\bf Step 5};

$\bullet$ if both determinants are non-zero and have the same sign, then we 
conclude that $f(t)\cap f(s)=\varnothing$, write nothing to the auxiliary file, 
and go to {\bf Step 5};

$\bullet$ if both determinants are non-zero and have different signs, then 
calculate three more determinants
$6\operatorname{Vol}({\bf z}_1, {\bf z}_2, {\bf y}_1, {\bf y}_2)$,
$6\operatorname{Vol}({\bf z}_1, {\bf z}_2, {\bf y}_2, {\bf y}_3)$,
$6\operatorname{Vol}({\bf z}_1, {\bf z}_2, {\bf y}_3, {\bf y}_1)$
(see formulas (\ref{eq2}), (\ref{eq4})), and proceed as follows:

{} \quad $\diamond$ if at least one of these three determinants is equal to zero, 
then write the message ``The question whether $f(t)$ and~$f(s)$ do intersect 
requires additional study'' to the auxiliary file, and go to {\bf Step 5};
  
{} \quad $\diamond$ if all three determinants are non-zero and have the same sign, 
then we conclude that $f(t)\cap f(s)\neq\varnothing$, write the message 
``$f(t)$ and $f(s)$ intersect'' to the auxiliary file, and go to {\bf Step 5};

{} \quad $\diamond$ if all three determinants are non-zero but not all have the 
same sign, then we conclude that $f(t)\cap f(s)=\varnothing$, write nothing to the
auxiliary file, and go to {\bf Step 5.}

{\bf Step 4.}
[The case when $t$ and $s$ have exactly one common point.]

Calculate two determinants
$6\operatorname{Vol}({\bf y}_1, {\bf y}_2, {\bf y}_3, {\bf z}_1)$ and
$6\operatorname{Vol}({\bf y}_1, {\bf y}_2, {\bf y}_3, {\bf z}_2)$ 
(see formulas (\ref{eq2}), (\ref{eq3})).
Taking into account that at least one of them is equal to zero, proceed as follows:

$\bullet$ if one of these determinants is not equal to zero, then we conclude 
that~$f(s)$ does not lie in the plane containing~$f(t)$ and the pair $(s,t)$ does 
not contribute to self-intersections of~$f(M)$; in this case write nothing to the
auxiliary file, and go to {\bf Step 5.}

$\bullet$ 
if both of these determinants are equal to zero, then we conclude that~ $f(s)$ 
lies in the plane containing~$f(t)$; in this case write the message ``The question
whether $f(t)$ and~$f(s)$ do intersect requires additional study'' to the auxiliary 
file, and go to {\bf Step 5}.

{\bf Step 5.} 
[Move to the next pair $(t,s)$ or terminate the algorithm.]

Do this:

$\bullet$ 
if $(t,s)$ is not the maximal element of the linearly ordered set~$(T\times S, <)$, 
then replace $(t,s)$ with the next element in the order and go to {\bf Step 2};

$\bullet$ if $(t,s)$ is the maximal element of~$(T\times S, <)$, then output the
auxiliary file and finish the algorithm.

The main result of \S~\ref{sec5} is the following lemma.

\begin{lemma}\label{lemma3}
Modified Steffen polyhedron $\mathscr{M}$ is combinatorially equivalent to some 
partition of the sphere, has only triangular faces, has no self-intersections, 
and is flexible.
\end{lemma}

\begin{proof}
It follows directly from the construction of~$\mathscr{M}$ that it is 
combinatorially equivalent to the triangulation of the sphere shown 
in Fig.~\ref{f4}, and is flexible.
 
We checked the absence of self-intersections using the algorithm described above,
implemented in \textit{Mathematica} using exclusively symbolic calculations.
Table~\ref{tab1} provides us with the list of vertices of~$\mathscr{M}$
and expressions in radicals for all their coordinates.
Lists of edges and faces of~$\mathscr{M}$ can be easily compiled using Fig.~\ref{f4}.
The calculation time was about 0.01 of a second. 
Our algorithm found no self-intersections of~$\mathscr{M}$ or ``exceptional'' cases
requiring additional study.
Based on this, we consider Lemma~\ref{lemma3} proven.
\end{proof}

Although this is not necessary for solution of Problem~\ref{probl1}, we have 
verified that Steffen polyhedron $\mathscr{S}$ has no self-intersections.
To do this, we applied to~$\mathscr{S}$ reasoning and calculations similar to 
those given in the proof of Lemma~\ref{lemma3}.
Our algorithm found no self-intersections of~$\mathscr{S}$ or ``exceptional'' cases
requiring additional study.

\section{Construction of polyhedron $\mathscr{P}$}\label{sec6} 
 
Let $\mathscr{M}$ be the modified Steffen polyhedron constructed in \S~\ref{sec4}.
Let $K:\mathbb{R}^3\to\mathbb{R}^3$ be the rotation around the $z$-axis by 
$90^{\circ}$ such that $K(T_3)=T_2$.
In the coordinate system constructed in \S~\ref{sec4} and associated with 
$\mathscr{M}$, the rotation $K$ corresponds to the matrix
\begin{equation}
K=\begin{pmatrix} 
0 & -1 & 0 \\ 
1 & 0 & 0 \\
0 & 0 & 1  
\end{pmatrix}. \label{eq5}
\end{equation}  
Let us consider the images of~$\mathscr{M}$ under the action of the maps 
$K$, $K^2=K\circ K$ and $K^3=K\circ K^2$, and glue them along the coincided faces.
We denote the resulting polyhedron by $\mathscr{P}$.
It has 26 vertices, 72 edges and 48 faces.

Informally speaking, the main idea of constructing~$\mathscr{P}$ is to ``surround'' 
the edge $T_1T_4$ of~$\mathscr{M}$ with isometric copies of~$\mathscr{M}$ so that 
$T_1T_4$ is no longer an edge of~$\mathscr{P}$ and, thus, the question ``whether
the dihedral angle attached to $T_1T_4$ is constant'' makes no sense.
However, for what follows, such a general idea of~$\mathscr{P}$ will not be 
sufficient.
Therefore, we are forced to go into detail. 

We denote the vertices of polyhedron $K(\mathscr{M})$ with the same letters 
as the corresponding vertices of~$\mathscr{M}$, but we provide them with a prime.
Moreover, we use the same convention as in \S~\ref{sec3}, namely: if a vertex 
of~$K(\mathscr{M})$ coincides with a vertex of~ $\mathscr{M}$ so that they must 
be glued together and must be considered as a single vertex of~$\mathscr{P}$, 
then for this vertex we always use the notation that it had in~$\mathscr{M}$.
For example, vertex $T'_1=K(T_1)$ coincides with $T_1$ and in~$\mathscr{P}$ is 
denoted by $T_1$; vertex $T'_4=K(T_4)$ coincides with $T_4$ and is denoted by $T_4$;
finally, vertex $T'_3=K(T_3 )$ coincides with $T_2$ and is denoted by $T_2$.
Using these notations we can say that $\mathscr{M}$ and $K(\mathscr{M})$ are glued 
along the coinciding faces $T_1T_2T_4$ and $T'_1T'_3T'_4$, which after gluing 
``disappear'' so that they are not faces of~$\mathscr{P}$.

Similarly, we denote the vertices of polyhedra~$K^2(\mathscr{M})$ 
and~$K^3(\mathscr{M})$ with the same letters as the corresponding vertices 
of~$\mathscr{M}$, but we provide their with two and three primes, respectively. 
Moreover, if some vertex of~$\mathscr{P}$ appears in our constructions several 
times, then we assign to it the name that contains the minimum number of primes.
For example, in this notation~$K^3(\mathscr{M})$ and~$\mathscr{M}$ are glued 
along the coincided faces $T'''_1T'''_2T'''_4$ and~$T_1T_3T_4$; and vertex~$T'''_2$ 
of~$K^3(\mathscr{M})$ receives the designation $T_3$ in~$\mathscr{P}$.

Using these notations and Table~\ref{tab1} we can easily find the coordinates 
of any vertex of~$\mathscr{P}$. For example,
\begin{eqnarray*}
A'''_2&=&K^3(A_2)=
\begin{pmatrix} 
0 & 1 & 0 \\ 
-1 & 0 & 0 \\
0 & 0 & 1  
\end{pmatrix}A_2= \notag \\
&=&\biggl(\frac{\omega_2\sqrt{2}+2\sqrt{\rho}}{15573468962},
-\frac{\omega_1\sqrt{2}-2(200-33\sqrt{31})\sqrt{\rho}}{1230304047998},
\frac{167\omega_3\sqrt{2}+22\sqrt{\rho}}{15573468962\sqrt{167}}\biggr)
\approx (8.89, 1.19, 4.72),
\end{eqnarray*}
where the expressions $\omega_1$, $\omega_2$, $\omega_3$ and $\rho$ 
are defined in the caption to Table~\ref{tab1}.

The list of all vertices of~$\mathscr{P}$ along with their coordinates 
is given in Table~\ref{tab2}.
\begin{table}[h]
\begin{tabular}{llll}
\hline
   &\hphantom{A}{\small$x$-coordinate} 
   &\hphantom{A}{\small$y$-coordinate} 
   &\hphantom{A}{\small$z$-coordinate}\\
\hline
$T_1$   & {\small\hphantom{$-$}$0$} & {\small\hphantom{$-$}$0$} 
        & {\small\hphantom{$-$}$\frac{\sqrt{167}}{\sqrt{2}}\approx 9.13 
        $} \\
$T_2$   & {\small\hphantom{$-$}$0$} 
        & {\small\hphantom{$-$}$\frac{11}{\sqrt{2}}\approx 7.77 
        $} 
        & {\small\hphantom{$-$}$0$} \\
$T_3$   & {\small\hphantom{$-$}$\frac{11}{\sqrt{2}}\approx 7.77 
        $} 
        & {\small\hphantom{$-$}$0$} & {\small\hphantom{$-$}$0$} \\
$T_4$   & {\small\hphantom{$-$}$0$} & {\small\hphantom{$-$}$0$} 
        & {\small$-\frac{\sqrt{167}}{\sqrt{2}}\approx -9.13 
        $} \\ 
$C_2$   & {\small\hphantom{$-$}$\frac{11+3\sqrt{31}}{2\sqrt{2}}\approx 9.79 
        $} 
        & {\small\hphantom{$-$}$\frac{11+3\sqrt{31}}{2\sqrt{2}}\approx 9.79 
        $} 
        & {\small\hphantom{$-$}$0$} \\
$A_2$   & {\small\hphantom{$-$}$\frac{\omega_1\sqrt{2}-
        2(200-33\sqrt{31})\sqrt{\rho}}{1230304047998}\approx -1.19$} 
        & {\small\hphantom{$-$}$\frac{\omega_2\sqrt{2}+2\sqrt{\rho}}{15573468962}
        \approx 8.89 
        $} 
        & {\small\hphantom{$-$}$\frac{167\omega_3\sqrt{2}+
        22\sqrt{\rho}}{15573468962\sqrt{167}}\approx 4.72 
        $} \\
$B_1$   & {\small\hphantom{$-$}$\frac{\omega_2\sqrt{2}-
        2\sqrt{\rho}}{15573468962}\approx 2.79 
        $} 
& {\small\hphantom{$-$}$\frac{\omega_1\sqrt{2}+
        2(200-33\sqrt{31})\sqrt{\rho}}{1230304047998}\approx 0.05
        $}
& {\small\hphantom{$-$}$\frac{167\omega_3\sqrt{2}-
        22\sqrt{\rho}}{15573468962\sqrt{167}}\approx -0.46 
        $} \\
$\overline{A}_2$    & {\small\hphantom{$-$}$\frac{\omega_1\sqrt{2}+
        2(200-33\sqrt{31})\sqrt{\rho}}{1230304047998}\approx 0.05
        $} 
        & {\small\hphantom{$-$}$\frac{\omega_2\sqrt{2}-
        2\sqrt{\rho}}{15573468962}\approx 2.79 
        $}  
& {\small$-\frac{167\omega_3\sqrt{2}-
        22\sqrt{\rho}}{15573468962\sqrt{167}}\approx 0.46 
        $} \\
$\overline{B}_1$    & {\small\hphantom{$-$}$\frac{\omega_2\sqrt{2}+
        2\sqrt{\rho}}{15573468962}\approx 8.89 
        $} 
        & {\small\hphantom{$-$}$\frac{\omega_1\sqrt{2}-
        2(200-33\sqrt{31})\sqrt{\rho}}{1230304047998}\approx -1.19
        $} 
        & {\small$-\frac{167\omega_3\sqrt{2}+
        22\sqrt{\rho}}{15573468962\sqrt{167}}\approx -4.72 
        $} \\
        \hline 
$T'_2$  & {\small$-\frac{11}{\sqrt{2}}\approx -7.77 
        $} 
        & {\small\hphantom{$-$}$0
        $} 
        & {\small\hphantom{$-$}$0$} \\
$C'_2$  & {\small$-\frac{11+3\sqrt{31}}{2\sqrt{2}}\approx -9.79 
        $} 
        & {\small\hphantom{$-$}$
        \frac{11+3\sqrt{31}}{2\sqrt{2}}\approx 9.79        
        $} 
        & {\small\hphantom{$-$}$0$} \\
$A'_2$  & {\small$-\frac{\omega_2\sqrt{2}+2\sqrt{\rho}}{15573468962}
        \approx -8.89 
        $} 
        & {\small\hphantom{$-$}$\frac{\omega_1\sqrt{2}-
        2(200-33\sqrt{31})\sqrt{\rho}}{1230304047998}\approx -1.19
        $} 
        & {\small\hphantom{$-$}$\frac{167\omega_3\sqrt{2}+
        22\sqrt{\rho}}{15573468962\sqrt{167}}\approx 4.72 
        $} \\
$B'_1$   & {\small$-\frac{\omega_1\sqrt{2}+
        2(200-33\sqrt{31})\sqrt{\rho}}{1230304047998}\approx -0.05
        $}
        & {\small\hphantom{$-$}$\frac{\omega_2\sqrt{2}-
        2\sqrt{\rho}}{15573468962}\approx 2.79 
        $} 
        & {\small\hphantom{$-$}$\frac{167\omega_3\sqrt{2}-
        22\sqrt{\rho}}{15573468962\sqrt{167}}\approx -0.46  
        $} \\
$\overline{A}'_2$    & {\small$-\frac{\omega_2\sqrt{2}-
        2\sqrt{\rho}}{15573468962}\approx -2.79 
		$} 
        & {\small\hphantom{$-$}$\frac{\omega_1\sqrt{2}+
        2(200-33\sqrt{31})\sqrt{\rho}}{1230304047998}\approx 0.05
        $}          
		& {\small$-\frac{167\omega_3\sqrt{2}-
        22\sqrt{\rho}}{15573468962\sqrt{167}}\approx 0.46 
        $} \\
$\overline{B}'_1$    & {\small$-\frac{\omega_1\sqrt{2}-
        2(200-33\sqrt{31})\sqrt{\rho}}{1230304047998}\approx 1.19
        $} 
        & {\small\hphantom{$-$}$\frac{\omega_2\sqrt{2}+
        2\sqrt{\rho}}{15573468962}\approx 8.89 
        $} 
        & {\small$-\frac{167\omega_3\sqrt{2}+
        22\sqrt{\rho}}{15573468962\sqrt{167}}\approx -4.72 
        $} \\
\hline
$T''_2$  & {\small$0
        $} 
        & {\small$-\frac{11}{\sqrt{2}}\approx -7.77 
        $}         
        & {\small\hphantom{$-$}$0$} \\
$C''_2$  & {\small$-\frac{11+3\sqrt{31}}{2\sqrt{2}}\approx -9.79        
        $}  
        & {\small$-\frac{11+3\sqrt{31}}{2\sqrt{2}}\approx -9.79 
        $}
        & {\small\hphantom{$-$}$0$} \\
$A''_2$  & {\small$-\frac{\omega_1\sqrt{2}-
        2(200-33\sqrt{31})\sqrt{\rho}}{1230304047998}\approx 1.19
        $}  
        & {\small$-\frac{\omega_2\sqrt{2}+2\sqrt{\rho}}{15573468962}
        \approx -8.89  
        $}
        & {\small\hphantom{$-$}$\frac{167\omega_3\sqrt{2}+
        22\sqrt{\rho}}{15573468962\sqrt{167}}\approx 4.72 
        $} \\
$B''_1$  & {\small$-\frac{\omega_2\sqrt{2}-
        2\sqrt{\rho}}{15573468962}\approx -2.79 
        $} 
        & {\small$-\frac{\omega_1\sqrt{2}+
        2(200-33\sqrt{31})\sqrt{\rho}}{1230304047998}\approx -0.05
        $}
        & {\small\hphantom{$-$}$\frac{167\omega_3\sqrt{2}-
        22\sqrt{\rho}}{15573468962\sqrt{167}}\approx -0.46  
        $} \\
$\overline{A}''_2$    & {\small$-\frac{\omega_1\sqrt{2}+
        2(200-33\sqrt{31})\sqrt{\rho}}{1230304047998}\approx -0.05
        $}          
        & {\small$-\frac{\omega_2\sqrt{2}-
        2\sqrt{\rho}}{15573468962}\approx -2.79 
		$} 
		& {\small$-\frac{167\omega_3\sqrt{2}-
        22\sqrt{\rho}}{15573468962\sqrt{167}}\approx 0.46 
        $} \\
$\overline{B}''_1$    & {\small$-\frac{\omega_2\sqrt{2}+
        2\sqrt{\rho}}{15573468962}\approx -8.89 
        $} 
        & {\small$-\frac{\omega_1\sqrt{2}-
        2(200-33\sqrt{31})\sqrt{\rho}}{1230304047998}\approx 1.19
        $} 
        & {\small$-\frac{167\omega_3\sqrt{2}+
        22\sqrt{\rho}}{15573468962\sqrt{167}}\approx -4.72 
        $} \\
\hline
$C'''_2$  & {\small\hphantom{$-$}$\frac{11+3\sqrt{31}}{2\sqrt{2}}\approx 9.79
        $}
        & {\small$-\frac{11+3\sqrt{31}}{2\sqrt{2}}\approx -9.79         
        $} 
        & {\small\hphantom{$-$}$0$} \\
$A'''_2$  & {\small\hphantom{$-$}$\frac{\omega_2\sqrt{2}+2\sqrt{\rho}}{15573468962}
        \approx 8.89  
        $}
        & {\small$-\frac{\omega_1\sqrt{2}-
        2(200-33\sqrt{31})\sqrt{\rho}}{1230304047998}\approx 1.19
        $}  
        & {\small\hphantom{$-$}$\frac{167\omega_3\sqrt{2}+
        22\sqrt{\rho}}{15573468962\sqrt{167}}\approx 4.72 
        $} \\
$B'''_1$  & {\small\hphantom{$-$}$\frac{\omega_1\sqrt{2}+
        2(200-33\sqrt{31})\sqrt{\rho}}{1230304047998}\approx 0.05
        $}
        & {\small$-\frac{\omega_2\sqrt{2}-
        2\sqrt{\rho}}{15573468962}\approx -2.79 
        $} 
        & {\small\hphantom{$-$}$\frac{167\omega_3\sqrt{2}-
        22\sqrt{\rho}}{15573468962\sqrt{167}}\approx -0.46 
        $} \\
$\overline{A}'''_2$    & {\small\hphantom{$-$}$\frac{\omega_2\sqrt{2}-
        2\sqrt{\rho}}{15573468962}\approx 2.79 
		$} 
        & {\small$-\frac{\omega_1\sqrt{2}+
        2(200-33\sqrt{31})\sqrt{\rho}}{1230304047998}\approx -0.05 
        $}  
		& {\small$-\frac{167\omega_3\sqrt{2}-
        22\sqrt{\rho}}{15573468962\sqrt{167}}\approx 0.46 
        $} \\
$\overline{B}'''_1$    & {\small\hphantom{$-$}$\frac{\omega_1\sqrt{2}-
        2(200-33\sqrt{31})\sqrt{\rho}}{1230304047998}\approx -1.19
        $} 
        & {\small$-\frac{\omega_2\sqrt{2}+
        2\sqrt{\rho}}{15573468962}\approx -8.89 
        $} 
        & {\small$-\frac{167\omega_3\sqrt{2}+
        22\sqrt{\rho}}{15573468962\sqrt{167}}\approx -4.72 
        $} \\
\hline
\newline
\end{tabular}
\caption{Exact and approximate values of the coordinates of the vertices 
of~$\mathscr{P}$. 
The expressions $\omega_1$, $\omega_2$, $\omega_3$ and $\rho$ 
are defined in the caption to Table~\ref{tab1}.
All decimal places in approximate values are correct.}\label{tab2}
\end{table}

We formulate the properties of~$\mathscr{P}$ related to our research 
in the following lemma.

\begin{lemma}\label{lemma4}
Polyhedron $\mathscr{P}$ is combinatorially equivalent to some triangulation 
of the sphere, has no self-intersections, and is flexible.
\end{lemma}

\begin{proof}
It follows directly from the construction of~$\mathscr{P}$ that it is 
combinatorially equivalent to a triangulation of the sphere and is flexible.

We checked the absence of self-intersections using the algorithm that was
described in~\S~\ref{sec5} and was implemented in \textit{Mathematica} using 
exclusively symbolic calculations.
Table~\ref{tab2} provides us with the list of vertices of~$\mathscr{P}$
and expressions in radicals for all their coordinates.
Lists of edges and faces of~$\mathscr{P}$ are given in Tables~\ref{tab3} 
and~\ref{tab4}, respectively.
The calculation time was less than 0.1 of a second. 
Our algorithm found no self-intersections of~$\mathscr{P}$ or ``exceptional'' cases
requiring additional study.
\end{proof}

\begin{table}[t]
\begin{tabular}{lllll}
\hline
    {\small $\mathscr{P}\cap\mathscr{M}$}\hphantom{AAAAAAA}
   &{\small $\mathscr{P}\cap K(\mathscr{M})$}\hphantom{AAAAAAA} 
   &{\small $\mathscr{P}\cap K^2(\mathscr{M})$}\hphantom{AAAAAAA} 
   &{\small $\mathscr{P}\cap K^3(\mathscr{M})$}\hphantom{AAAAAAA}
   & \hphantom{A} \\
\hline
 $T_1T_2$   &  $T_1T'_2$  &  $T_1T''_2$  &  \hphantom{A}$\varnothing$ 
 & \ref{lemma7}   \\
 $T_1T_3$   &  \hphantom{A}$\varnothing$ &  \hphantom{A}$\varnothing$   
 & \hphantom{A}$\varnothing$   & \ref{lemma7} \\
 $T_1A_2$   &  $T_1A'_2$  &  $T_1A''_2$  &  $T_1A'''_2$ & \ref{lemma5} \\
 $T_1B_1$   &  $T_1B'_1$  &  $T_1B''_1$  &  $T_1B'''_1$ & \ref{lemma5} \\
 $T_2T_4$   &  $T'_2T_4$  &  $T''_2T_4$  &  \hphantom{A}$\varnothing$ 
 & \ref{lemma7}   \\
 $T_2C_2$   &  $T'_2C'_2$ & $T''_2C''_2$ &  $T_3C'''_2$ & \ref{lemma6} \\
 $T_2A_2$   &  $T'_2A'_2$ & $T''_2A''_2$ &  $T_3A'''_2$ & \ref{lemma5}\\
 $T_2\overline{A}_2$   &  $T'_2\overline{A}'_2$  & $T''_2\overline{A}''_2$  
 & $T_3\overline{A}'''_2$ & \ref{lemma5} \\
 $T_3T_4$   &  \hphantom{A}$\varnothing$ &  \hphantom{A}$\varnothing$    
 & \hphantom{A}$\varnothing$  & \ref{lemma7}   \\
 $T_3C_2$   &  $T_2C'_2$  & $T'_2C''_2$  &  $T''_2C'''_2$ & \ref{lemma6} \\
 $T_3B_1$   &  $T_2B'_1$  & $T'_2B''_1$  &  $T''_2B'''_1$ & \ref{lemma5} \\
 $T_3\overline{B}_1$   &  $T_2\overline{B}'_1$  & $T'_2\overline{B}''_1$  &
 $T''_2\overline{B}'''_1$  & \ref{lemma5} \\
 $T_4\overline{A}_2$   &  $T_4\overline{A}'_2$  & $T_4\overline{A}''_2$  &
 $T_4\overline{A}'''_2$  & \ref{lemma5} \\
 $T_4\overline{B}_1$   &  $T_4\overline{B}'_1$  & $T_4\overline{B}''_1$  &
 $T_4\overline{B}'''_1$ & \ref{lemma5} \\
 $C_2A_2$   &  $C'_2A'_2$ & $C''_2A''_2$ & $C'''_2A'''_2$ & \ref{lemma5} \\
 $C_2B_1$   &  $C'_2B'_1$ & $C''_2B''_1$ & $C'''_2B'''_1$ & \ref{lemma5} \\
 $C_2\overline{A}_2$   &  $C'_2\overline{A}'_2$ & $C''_2\overline{A}''_2$  &
 $C'''_2\overline{A}'''_2$ & \ref{lemma5} \\
 $C_2\overline{B}_1$   &  $C'_2\overline{B}'_1$ & $C''_2\overline{B}''_1$  &
 $C'''_2\overline{B}'''_1$ & \ref{lemma5} \\
 $A_2B_1$   &  $A'_2B'_1$ & $A''_2B''_1$ & $A'''_2B'''_1$ & \ref{lemma5} \\
 $\overline{A}_2\overline{B}_1$   &  $\overline{A}'_2\overline{B}'_1$  & 
 $\overline{A}''_2\overline{B}''_1$  &  $\overline{A}'''_2\overline{B}'''_1$ 
 & \ref{lemma5}\\
 \hline
 \newline
\end{tabular}
\caption{List of all 72 edges of polyhedron~$\mathscr{P}$. 
The sign $\varnothing$ marks cells that are intentionally left empty because the 
edge corresponding to such a cell has already appeared (in another row) in one of 
the columns located to the left.
The right column contains the numbers of lemmas, from the proof of which it follows 
that the dihedral angle at each edge in this row is not constant.}\label{tab3}
\end{table}
\begin{table}[h]
\begin{tabular}{llll}
\hline
    {\small $\mathscr{P}\cap\mathscr{M}$}\hphantom{AAAAAAA}
   &{\small $\mathscr{P}\cap K(\mathscr{M})$}\hphantom{AAAAAAA} 
   &{\small $\mathscr{P}\cap K^2(\mathscr{M})$}\hphantom{AAAAAAA} 
   &{\small $\mathscr{P}\cap K^3(\mathscr{M})$}\hphantom{AAAAAAA}\\
\hline
 $T_1T_2A_2$   & $T_1T'_2A'_2$  & $T_1T''_2A''_2$  & $T_1T_3A'''_2$    \\
 $T_1T_3B_1$   & $T_1T_2B'_1$   & $T_1T'_2B''_1$   & $T_1T''_2B'''_1$  \\
 $T_1A_2B_1$   & $T_1A'_2B'_1$  & $T_1A''_2B''_1$  & $T_1A'''_2B'''_1$  \\
 $T_2T_4\overline{A}_2$   &  $T'_2T_4\overline{A}'_2$   & $T''_2T_4\overline{A}''_2$  &  $T_3T_4\overline{A}'''_2$ \\  
 $T_2C_2A_2$   & $T'_2C'_2A'_2$ & $T''_2C''_2A''_2$  &  $T_3C'''_2A'''_2$ \\
 $T_2C_2\overline{A}_2$   & $T'_2C'_2\overline{A}'_2$   & $T''_2C''_2\overline{A}''_2$  &  $T_3C'''_2\overline{A}'''_2$ \\
 $T_3T_4\overline{B}_1$   & $T_2T_4\overline{B}'_1$  & $T'_2T_4\overline{B}''_1$  & $T''_2T_4\overline{B}'''_1$  \\
 $T_3C_2B_1$   & $T_2C'_2B'_1$  & $T'_2C''_2B''_1$  & $T''_2C'''_2B'''_1$  \\
 $T_3C_2\overline{B}_1$   & $T_2C'_2\overline{B}'_1$  & $T'_2C''_2\overline{B}''_1$  &  $T''_2C'''_2\overline{B}'''_1$ \\
 $T_4\overline{A}_2\overline{B}_1$   & $T_4\overline{A}'_2\overline{B}'_1$  & $T_4\overline{A}''_2\overline{B}''_1$  &  $T_4\overline{A}'''_2\overline{B}'''_1$ \\
 $C_2A_2B_1$   & $C'_2A'_2B'_1$  & $C''_2A''_2B''_1$  & $C'''_2A'''_2B'''_1$  \\
 $C_2\overline{A}_2\overline{B}_1$   & $C'_2\overline{A}'_2\overline{B}'_1$  & $C''_2\overline{A}''_2\overline{B}''_1$  &  $C'''_2\overline{A}'''_2\overline{B}'''_1$ \\
 \hline
 \newline
\end{tabular}
\caption{List of all 48 faces of polyhedron~$\mathscr{P}$.}\label{tab4}
\end{table}

\section{Special flex $\mathscr{P}(t)$ of polyhedron $\mathscr{P}$}\label{sec7} 

Let us start by discussing the flex of the modified Steffen polyhedron~$\mathscr{M}$.

In order to define such a flex, we need to specify the positions of all vertices 
of~$\mathscr{M}$ as continuous functions of some parameter~$t$ so that the distance
between any two vertices connected by an edge of~$\mathscr{M}$ does not depend on~$t$,
and the distance between at least two vertices not connected by an edge is not constant.
We denote each of these functions by the same symbol as the corresponding vertex 
of~$\mathscr{M}$. 
For example, the vertex $C_2\in \mathscr{M}$ corresponds to the function $C_2(t)$.
The values of these functions for a fixed $t$ are declared to be the vertices of 
the polyhedron~$\mathscr{M}(t)$ and, by definition, we put that the correspondence
``the vertex of~$\mathscr{M}$ $\leftrightarrow$ the value of the function of the 
same name'' defines the combinatorial equivalence of polyhedra $\mathscr{M}$ and 
$\mathscr{M}(t)$.
Finally, since the choice of parameter~$t$ is arbitrary, we assume without loss of
generality that $\mathscr{M}(0)=\mathscr{M}$.

The functions corresponding to the vertices of $\mathscr{M}$ are defined as follows.

We assume that the vertices $T_j$, $j=1,\dots, 4$, do not change their positions 
in space, i.e., by definition, we assume $T_j(t)=T_j$ for all $t$, where $T_j$ has 
the coordinates given in Table~\ref{tab1}.
 
Since the vertex $C_2(t)$ is connected in $\mathscr{M}(t)$ by edges to the vertices
$T_2(t)\equiv T_2$ and $T_3(t)\equiv T_3$, then it lies on the circle $\gamma$,
determined by the following conditions: $\gamma$ lies in a plane perpendicular to 
the segment $T_2T_3$; the center of~$\gamma$ coincides with the middle point 
of~$T_2T_3$; the radius of~$\gamma$ is equal to 
$\sqrt{|A_1C_2|^2-|A_1B_2|^2/4}=3\sqrt{31}/2$.
It is the movement of~$C_2(t)$ along~$\gamma$ that sets the flex of~$\mathscr{M}$ 
(here we do not care whether~$\mathscr{M}(t)$ has self-intersections).
Indeed, the positions of the five vertices $C_2(t)$, $T_j(t)$ ($j=1,\dots, 4$) 
of~$\mathscr{M}(t)$ uniquely determine the position of every of its other four vertices 
$A_2(t)$, $B_1(t)$, $\overline{A}_2(t)$ and $\overline{B}_1(t)$, because each of the
latter vertices is connected in $\mathscr{M}(t)$ by three edges to some of the five 
vertices~$C_2 (t)$, $T_j(t)$ ($j=1,\dots, 4$).
Taking into account that the position of each vertex of~$\mathscr{M}(t)$ changes
continuously and is known for $t=0$ from Table~\ref{tab1}, we are convinced that 
the positions of~$A_2(t) $, $B_1(t)$, $\overline{A}_2(t)$ and $\overline{B}_1(t)$ are
uniquely determined by the positions of~$C_2(t)$, $T_j(t)$ ( $j=1,\dots, 4$) for 
all~$t$ sufficiently close to zero (namely, until the three edges mentioned above 
are not in the same plane).
That is why we say that the flex~$\mathscr{M}(t)$ of~$\mathscr{M}$ is defined by the position of~$C_2(t)$.

Let us discuss the flex of polyhedron~$\mathscr{P}$.

By construction, $\mathscr{P}$ is obtained by gluing four copies of~$\mathscr{M}$,
namely, by gluing together polyhedra 
$\mathscr{M}$, $K(\mathscr{M})$, $K^2 (\mathscr{M})$ and $K^3(\mathscr{M})$.
Here, as before, $K:\mathbb{R}^3\to\mathbb{R}^3$ is the rotation around $z$-axis by
$90^{\circ}$; in particular, the matrix of~$K$ in the coordinate system constructed 
in \S~\ref{sec4} which is associated with~$\mathscr{P}$ is given by~(\ref{eq5}).
From what was said at the beginning of~\S~\ref{sec7} it is clear that, 
similarly to the movement of vertex $C_2\in \mathscr{M}$ along the circle $\gamma$, 
each of the vertices $C'_2\in K(\mathscr{M})$, $C''_2\in K^2(\mathscr{M})$ and
$C'''_2\in K^3(\mathscr{M})$ can be moved along the corresponding circle $K(\gamma)$,
$K^2(\gamma)$, $K^3(\gamma)$ independently of each other and of the motion of~$C_2$.
After the positions of the vertices $C_2(t)$, $C'_2(t)$, $C''_2(t)$ and $C'''_2(t)$ 
are given (and taking into account that the positions of the vertices 
$T_1$, $T_2$, $T_3 $, $T_4$, $T'_2$ and $T''_2$ do not change during the flex and 
are known to us from Table~\ref{tab2}), we can uniquely find the positions of 
all other vertices of~$\mathscr{P}(t)$.
To reflect the possibility of independent changing the positions of~$C_2$, $C'_2$, 
$C''_2$ and $C'''_2$, we say that~$\mathscr{P}$ admits a 4-parameter flex.
 
However, in accordance with the definition of a flexible polyhedron given 
in~\S~\ref{sec1}, when proving Theorem~\ref{thrm1} we need some special flex 
$\mathscr{P}(t)$ of~$\mathscr{P}$, depending on one real parameter $t$.
We distinguish it from the 4-parameter flex of~$\mathscr{P}$ just described by 
the following three conditions:

$\bullet$ $\mathscr{P}(0)=\mathscr{P}$; 

$\bullet$ for $t=0$, the speed of $C_2(t)$ is equal to $(0,0,1)$;

$\bullet$ for all $t$, $\mathscr{P}(t)$ is invariant under the rotation 
$K:\mathbb{R}^3\to\mathbb{R}^3$ around $z$-axis by $90^{\circ }$.

We denote the vertices of~$\mathscr{P}(t)$ by analogy with the corresponding 
vertices of~$\mathscr{P}$. For example, the vertices $C_2(t)$ and $A'''_2 (t)$ 
of~$\mathscr{P}(t)$ correspond to the vertices $C_2=C_2(0)$ and $A'''_2=A '''_2 (0)$ 
of~$\mathscr{P}$.
 
Recall that Lemma~\ref{lemma3} from~\S~\ref{sec5} and Lemma~\ref{lemma4} 
from~\S~\ref{sec6} state that polyhedra $\mathscr{M}$ and $\mathscr{P}$ 
have no self-intersections. 
Now we draw the reader's attention to the fact that without any additional 
calculations we can state that for all $t$ sufficiently close to zero 
polyhedra $\mathscr{M}(t)$ and $\mathscr{P}(t)$ also do not have self-intersections.
To prove this, observe that non of the determinants (\ref{eq3}) and (\ref{eq4}), 
calculated in the proofs of Lemmas~\ref{lemma3} and~\ref{lemma4}, were equal to zero.
Hence, they are nonzero for all $t$ sufficiently close to zero.
Having fixed any of these~$t$ and repeating the proofs of Lemmas~\ref{lemma3} 
and~\ref{lemma4} as applied to~ $\mathscr{M}(t)$ and $\mathscr{P}(t)$, we come 
to the conclusion that both $\mathscr{M}(t)$ and $\mathscr{P}(t)$ have 
no self-intersections.
 
Let us begin to study whether~$\mathscr{P}(t)$ has a dihedral angle, the value 
of which remains constant for all $t$ sufficiently close to zero.
For different dihedral angles we will need different arguments.

\begin{lemma}\label{lemma5}
The dihedral angle at the edge $A_2(t)T_1(t)\equiv A_2(t)T_1$ of 
polyhedron $\mathscr{P}(t)$ is not constant as a function of $t$. 
\end{lemma}

\begin{proof}
In~$\mathscr{P}$, the edge $A_2T_1$ is incident to faces $A_2B_1T_1$ and $A_2B_2T_1$.
Both of these faces belong to the Bricard octahedron $\mathscr{B}$, which 
participated in the construction of the modified Steffen polyhedron $\mathscr{M}$, 
which, in turn, was used to construct polyhedron~$\mathscr{P}$ 
(see \S\S~\ref{sec2}, \ref{sec3}, \ref{sec4}, \ref{sec6}).
The flex $\mathscr{P}(t)$ of~ $\mathscr{P}$ obviously gives rise to the flex 
$\mathscr{B}(t)$ of the Bricard octahedron $\mathscr{B}$ (indeed, the length of 
each edge of~$\mathscr{B}(t)$ does not depend on $t$, while the length of the 
diagonal $C_2(t)T_1$ is obviously nonconstant as function of $t$).
According to Lemma~\ref{lemma1}, during the flex of the Bricard octahedron 
$\mathscr{B}$ the value of each of its dihedral angles does not remain constant.
Hence, the value of the dihedral angle at edge $A_2(t)T_1$, considered either as 
a dihedral angle of the Bricard octahedron $\mathscr{B}(t)$, or as a dihedral 
angle of~$\mathscr{P}(t)$, is also nonconstant. 
\end{proof}

Note that in Lemma~\ref{lemma5} the edge $A_2(t)T_1$ can be replaced by any edge 
of~$\mathscr{P}(t)$, which is incident to two faces of any Bricard octahedron
participating in the construction of~$\mathscr{P}$ (i.e., the Bricard octahedra 
$\mathscr{B}$ and $\overline{\mathscr{B}}$ in notation of~\S~\ref{sec3}, as 
well as their images under rotations $K $, $K^2$ and $K^3$). 
There are a total of 7 such edges on each part of~$\mathscr{P}$ which  corresponds 
to a single Bricard octahedron.
Hence, there are 56 such edges on~$\mathscr{P}$.
Each row of Table~\ref{tab3}, consisting entirely of such edges, ends with 
the number 5.
We use this notation to fix the fact that the dihedral angles at the edges in 
such rows are nonconstant follows from the proof of Lemma~\ref{lemma5}.
Similarly, with numbers 6 and 7 in the last column we mark those rows of 
Table~\ref{tab3} for which the inconstancy of dihedral angles at the edges in 
these rows follows from the proofs of Lemmas~\ref{lemma6} and~\ref{lemma7}, 
respectively.

So, we already know that the dihedral angles at 56 edges do not 
remain constant during the special flex of~$\mathscr{P}$.
Lemmas~\ref{lemma6} and~\ref{lemma7} are devoted to the study of the remaining 
16 dihedral angles.
But before we formulate and prove them, we need to do preparatory work.

Let ${\bf p}_{j}(t)$, $j=1,2,3,4$, be points which may change their position in space 
in such a way that the distance between any two of them, except perhaps 
${\bf p}_{3}(t)$ and ${\bf p}_{4}(t)$, do not depend on $t$.
Then, as is known (see, for example, \cite[Section 3.2]{CG22}),
their velocities ${\bf v}_{j}(t)$ satisfy the relations
\begin{equation}
({\bf p}_{j}(t)-{\bf p}_{k}(t))\cdot ({\bf v}_{j}(t)-{\bf v}_{k}(t))=0,
\qquad j,k=1,2,3,4; \ \{j,k\}\neq\{3,4\},
\label{eq6}
\end{equation}
where $\cdot$ denotes the standard scalar product in $\mathbb{R}^3$.
It is easy to understand that if the inequality 
\begin{equation}
({\bf p}_{3}(0)-{\bf p}_{4}(0))\cdot ({\bf v}_{3}(0)-{\bf v}_{4}(0))\neq 0
\label{eq7}
\end{equation}
holds true, then the derivative, calculated at $t=0$, of the dihedral angle 
between the triangles
${\bf p}_{1}(t), {\bf p}_{2}(t), {\bf p}_{3}(t)$ and
${\bf p}_{1}(t), {\bf p}_{2}(t), {\bf p}_{4}(t)$ 
is not equal to zero, which means the value of this dihedral angle is not constant as 
a function of $t$ on some sufficiently small interval in $\mathbb{R}$ containing $0$.

\begin{lemma}\label{lemma6}
Dihedral angles at edges $C_2(t)T_2(t)\equiv C_2(t)T_2$ and
$C_2(t)T_3(t)\equiv C_2(t)T_3$ of polyhedron $\mathscr{P}(t)$ are not constant as
functions of $t$ on some sufficiently small interval in $\mathbb{R}$ containing 0.
\end{lemma}

\begin{proof}
By definition, put
${\bf p}_{1}(t)=C_2(t)$,
${\bf p}_{2}(t)\equiv T_2$,
${\bf p}_{3}(t)=A_2(t)$,
${\bf p}_{4}(t)=\overline{A}_2(t)$, 
${\bf p}_{5}(t)\equiv T_1$,
${\bf p}_{6}(t)\equiv T_4$, 
and denote by ${\bf v}_j(t)$, $j=1,\dots,6$, the velocity of the point~${\bf p}_j(t)$.
To prove the statement of Lemma~\ref{lemma6} it is sufficient to prove 
inequality~(\ref{eq7}) for $t=0$.
We do this using symbolic computation in \textit{Mathematica}.

The coordinates of the points ${\bf p}_{1}(0)=C_2$, ${\bf p}_{2}(0)=T_2$, 
${\bf p}_{3}(0)=A_2$, ${\bf p}_{4}(0)=\overline{A}_2$, 
${\bf p}_{5}(0)=T_1$ and ${\bf p}_{6}(0)=T_4$
in radicals we take from Table~\ref{tab1}.
The velocity vectors ${\bf v}_1(0)=(0,0,1)$ and
${\bf v}_2(0)={\bf v}_5(0)={\bf v}_6(0)=(0,0,0)$ are known to us by construction.

The components of the velocity vector ${\bf v}_3(0)$ must satisfy the following 
system of linear algebraic equations, each of which is similar to equation (\ref{eq6}):
\begin{equation}
\begin{cases}
({\bf p}_{1}(0)-{\bf p}_{3}(0))\cdot ({\bf v}_{1}(0)-{\bf v}_{3}(0))=0,\\
({\bf p}_{2}(0)-{\bf p}_{3}(0))\cdot ({\bf v}_{2}(0)-{\bf v}_{3}(0))=0,\\
({\bf p}_{5}(0)-{\bf p}_{3}(0))\cdot ({\bf v}_{5}(0)-{\bf v}_{3}(0))=0.\\
\end{cases}\label{eq8}
\end{equation}
Using \textit{Mathematica}, we solve (\ref{eq8}) and find the values of the 
components of~${\bf v}_{3}(0)$ in radicals.
However, they are too long to be written here.
So, we present their approximate numerical values only:
${\bf v}_{3}(0)\approx(-0.4602, -0.1074, -0.0914)$, 
where all decimal places are correct (i.e., we do not apply rounding rules).

Similarly, the components of the velocity vector ${\bf v}_4(0)$ must satisfy 
the system of equations
\begin{equation}
\begin{cases}
({\bf p}_{1}(0)-{\bf p}_{4}(0))\cdot ({\bf v}_{1}(0)-{\bf v}_{4}(0))=0,\\
({\bf p}_{2}(0)-{\bf p}_{4}(0))\cdot ({\bf v}_{2}(0)-{\bf v}_{4}(0))=0,\\
({\bf p}_{6}(0)-{\bf p}_{4}(0))\cdot ({\bf v}_{6}(0)-{\bf v}_{4}(0))=0.\\
\end{cases}\label{eq9}
\end{equation}
Using \textit{Mathematica}, we solve (\ref{eq9}) and find the values of the 
components of~${\bf v}_{4}(0)$ in radicals.
As before, they are too long to be written here and 
we present their approximate numerical values only:
${\bf v}_{4}(0)\approx(-0.0470, -0.0004, 0.0004)$.

Finally, using \textit{Mathematica} we evaluate the expression
\begin{equation}
({\bf p}_{3}(0)-{\bf p}_{4}(0))\cdot ({\bf v}_{3}(0)-{\bf v}_{4}(0)).
\label{eq10}
\end{equation}
The result, written in radicals, is too long to be given here; but, again 
with the help of \textit{Mathematica}, we are convinced that it is non-zero 
in the corresponding extension of the field of rational numbers.
The numerical value of expression (\ref{eq10}) is equal to $-0.5253$; 
this gives us additional confidence in the adequacy of our symbolic calculations. 
Thus, the statement of Lemma~\ref{lemma6} about the inconstancy of the dihedral 
angle at the edge $C_2(t)T_2$ of $\mathscr{P}(t)$ is proven.

The statement of Lemma~\ref{lemma6} about the inconstancy of the dihedral angle 
at the edge $C_2(t)T_3$ of $\mathscr{P}(t)$ can be proven in the same way as 
we have just proved a similar statement for the angle at the edge~$C_2(t)T_2$. 
But we prefer to avoid these additional calculations by noting that 
the rotation~$L$ defined in the statement of Lemma~\ref{lemma2} maps
the polyhedron $\mathscr{P}$ onto itself.
To present this reasoning in more detail, we introduce the notation
${\bf q}_{1}(t)=C_2(t)$,
${\bf q}_{2}(t)\equiv T_3$,
${\bf q}_{3}(t)=\overline{B}_1(t)$,
${\bf q}_{4}(t)={B}_1(t)$, 
${\bf q}_{5}(t)\equiv T_4$,
${\bf q}_{6}(t)\equiv T_1$.
Denote by ${\bf w}_j(t)$, $j=1,\dots,6$, the velocity of the point ${\bf q}_{j}(t)$.
Lemma~\ref{lemma2} implies that $L({\bf q}_{j}(0))={\bf p}_{j}(0)$ 
for all $j=1,\dots,6$. 
We have already reflected this fact in the Table~\ref{tab1}.
Taking into account that ${\bf w}_1(0)=(0,0,1)$ and
${\bf w}_2(0)={\bf w}_5(0)={\bf w}_6(0)=(0,0,0)$, 
and writing equations for finding ${\bf w}_{3}(0)$ and ${\bf w}_{4}(0)$ 
by analogy with (\ref{eq8}) and (\ref{eq9}), we obtain 
$L({\bf w}_{j}(0))=-{\bf v}_{j}(0)$, $j=1,\dots,6$.
Hence, 
$({\bf q}_{3}(0)-{\bf q}_{4}(0))\cdot ({\bf w}_{3}(0)-{\bf w}_{4}(0))=
-({\bf p}_{3}(0)-{\bf p}_{4}(0))\cdot ({\bf v}_{3}(0)-{\bf v}_{4}(0))\neq 0$.
As we already know, this inequality implies that the dihedral angle at 
the edge~$C_2(t)T_3$ is nonconstant.
\end{proof}

Note that in Lemma~\ref{lemma6} the edges $C_2T_2$ and $C_2T_3$ can be replaced 
by their images under the action of rotations~$K$, $K^2$ and $K^3$.
This means that Lemma \ref{lemma6} guarantees us that the eight dihedral angles 
of~$\mathscr{P}(t)$ are not constant as functions of~$t$.
This fact is reflected in the right column of Table~\ref{tab3}.

\begin{lemma}\label{lemma7}
Dihedral angles at edges $T_1T_3$ and $T_2T_4$ of polyhedron $\mathscr{P}(t)$ 
are not constant as functions of $t$ on some sufficiently small interval 
in $\mathbb{R}$ containing 0.
\end{lemma}

\begin{proof}
Lemma~\ref{lemma7} can be proven similarly to Lemma~\ref{lemma6}. 
The proof is left to the reader.
\end{proof}

Note that in Lemma~\ref{lemma7} the edges $T_1T_3$ and $T_2T_4$ can be replaced 
by their images under the action of rotations~$K$, $K^2$ and $K^3$.
This means that Lemma \ref{lemma7} guarantees us that the eight dihedral angles 
of~$\mathscr{P}(t)$ are not constant as functions of~$t$.
This fact is reflected in the right column of Table~\ref{tab3}.

Thus, Lemmas ~\ref{lemma5}, \ref{lemma6} and~\ref{lemma7} imply that all 72 
dihedral angles of~$\mathscr{P}(t)$ are not constant as functions of $t$ 
on some sufficiently small interval in $\mathbb{R}$ containing 0.

\section{Proof of Theorem~\ref{thrm1} and concluding remarks}\label{sec8} 

\begin{proof} All components of the proof of Theorem~\ref{thrm1} 
have already appeared in the previous Sections. 
We just need to put them together.

Let $\mathscr{P}$ be the polyhedron constructed in~\S~\ref{sec6}.
According to Lemma~\ref{lemma3}, $\mathscr{P}$ has no self-intersections,
is homeomorphic to a sphere, has only triangular faces, and is flexible.
Thus, statement (i) of Theorem~\ref{thrm1} is fulfilled for~$\mathscr{P}$.

In \S~\ref{sec7} a special flex $\mathscr{P}(t)$ of~$\mathscr{P}$ is constructed.
According to Lemmas ~\ref{lemma5}, \ref{lemma6} and \ref{lemma7}, there is a 
sufficiently small neighborhood of 0 in $\mathbb{R}$ in which none of the 
dihedral angles of~$\mathscr{P}(t)$ is constant as a function of $t$.
Reducing the range of~$t$ to the segment $[0,\beta)$ completely contained in 
this neighborhood, we obtain the family of polyhedra 
$\{\mathscr{P}(t)\}_{t\in [0, \beta)}$, 
whose existence proves statement (ii) of Theorem~\ref{thrm1}.
\end{proof} 

Let us formulate two open Problems~\ref{probl2}
and~\ref{probl3}, closely related to Problem~\ref{probl1}.

\begin{probl}\label{probl2}
Is there a one-parametric flexible polyhedron in $\mathbb{R}^3$, 
without boundary and without self-intersections, for which all dihedral 
angles change during the flex?
\end{probl} 

The concept of a one-parametric flexible polyhedron is intuitively obvious. 
This is exactly how it was used in~\cite{Sh15}.
The definition of a $p$-parametric flexible polyhedron requires clarification. 
It is strictly formulated in \cite{MS02}.
For our purposes, it is sufficient to have an intuitive understanding that 
flexible Steffen polyhedron $\mathscr{S}$ is one-parametric, and flexible 
polyhedron $\mathscr{P}$ constructed in \S~\ref{sec6} is 4-parametric.
 
Problem~\ref{probl2} seems interesting to us due to the fact that in mathematics 
there are many situations where the behavior of an object changes greatly 
depending on whether it depends on one parameter or on several.

It is obvious that the statement ``small diagonal $AB$ of a polyhedron $\mathscr{R}$ 
does not change during the flex'' is equivalent to the statement ``the dihedral 
angle at the common edge of those two faces of~$\mathscr{R}$ which contain $A$ and $B$ 
does not change during the flex''.
Therefore, the following Problem~\ref{probl3} is a generalization of 
Problem~\ref{probl1}.

\begin{probl}\label{probl3}
Is there a flexible polyhedron in $\mathbb{R}^3$, without boundary and without 
self-intersections, for which all diagonals change during the flex?
\end{probl} 

Examples show that Problem~\ref{probl3} is non-trivial.
Indeed, Bricard octahedra have self-intersections and have no diagonals whose 
lengths do not change during the flex (see Lemma~\ref{lemma1}).
On the other hand, Steffen flexible polyhedron $\mathscr{S}$, modified Steffen 
flexible polyhedron $\mathscr{M}$, and flexible polyhedron~$\mathscr{P}$ constructed 
in \S~\ref{sec6} have no self-intersections, but have diagonals whose lengths do 
not change during the flex (namely, $T_2T_3$ for $\mathscr{S}$ and $\mathscr{M}$; 
and, for example, $T_1T_4$ and $T_2T_3$ for $\mathscr{P}$).
 
An additional interest to Problem~\ref{probl3} is added by the fact that, 
unlike Problems~\ref{probl1} and~\ref{probl2}, it makes sense not only 
for orientable polyhedra.
Note also that Problem~\ref{probl3} can be considered as a special case of the 
problem of which lengths of diagonals of a flexible polyhedron must be fixed 
so that it ceases to be flexible.
Various aspects of the latter problem have been studied, for example, 
in~\cite{MS02}, \cite{Sa02}, \cite{Sa11}.

Note that analogues of Problems~\ref{probl1}--~\ref{probl3} can be posed not only 
in~$\mathbb{R}^3$, but also in any space of constant curvature of dimension $\geq 3$, 
as well as in spaces with an indefinite metric.

In conclusion, let us clarify that all symbolic calculations performed in the 
preparation of this article were performed using the computer software system
\textit{Mathematica} 12.1 \cite{Wo99}, license 3322--8225.

\subsection*{Acknowledgments} 
The authors thank N.P.~Dolbilin, A.A.~Gaifullin and I.Kh.~Sabitov for their 
interest in this study and useful discussions.

\subsection*{Funding}
The work was carried out within the State Task to the Sobolev Institute 
of Mathematics (namely, V.A. Alexandrov was supported by Project FWNF--2022--0006 and 
E.P. Volokitin was supported by Project FWNF--2022--0005).

\end{document}